\documentclass[12pt]{amsart}
\usepackage[osf,sc]{mathpazo}
\usepackage{amssymb}

\usepackage{geometry}
\usepackage{bm}
\usepackage{bbm}
\usepackage{graphicx}
\usepackage{hyperref}
\hypersetup{
    colorlinks=true, 
    linktoc=all,     
    linkcolor=blue,
    citecolor=red,
    filecolor=black,
    urlcolor=blue	
}
\usepackage{enumerate}
\usepackage[inline]{enumitem}
\makeatletter
\newcommand{\inlineitem}[1][]{%
\ifnum\enit@type=\tw@
    {\descriptionlabel{#1}}
  \hspace{\labelsep}%
\else
  \ifnum\enit@type=\z@
       \refstepcounter{\@listctr}\fi
    \quad\@itemlabel\hspace{\labelsep}%
\fi} \makeatother
\newcommand{\gl}{\lambda}
\newcommand{\gm}{\mu}

\newcommand{\gp}{\pi}

\newcommand{\Gl}{\Lambda}

\newcommand{\Gs}{\Sigma}

\newcommand{\Gom}{\Omega}

\newcommand{\subs}{\subset}
\newcommand{\sups}{\supset}

\newcommand{\sbnq}{\subsetneq}

\newcommand{\bs}{\backslash}

\newcommand{\nin}{\notin}

\newcommand{\ti}{\tilde}

\newcommand{\mbb}{\mathbb}

\newcommand{\mcl}{\mathcal}

\newcommand{\ol}{\overline}

\newcommand{\us}{\underset}
\newcommand{\os}{\overset}

\newcommand{\lra}{\longrightarrow}

\newcommand{\I}{\mcl I}

\newcommand{\N}{\mbb N}
\newcommand{\Z}{\mbb Z}
\newcommand{\R}{\mcl R}

\newcommand{\Ra}{\Rightarrow}

\newcommand{\lmt}{\longmapsto}
\newcommand{\eqdef}{\overset{\mathrm{def}}{=\joinrel=}}

\newcommand{\equ}[1]{%
\begin{equation*}
#1
\end{equation*}
}
\newcommand{\equa}[1]{%
\begin{equation*}
\begin{aligned}
#1
\end{aligned}
\end{equation*}
}

\DeclareMathOperator{\Det}{Det}

\DeclareMathOperator{\Diag}{Diag}
\newcommand{\mattwo}[4]{%
\begin{pmatrix}
  #1 & #2\\ #3 & #4
\end{pmatrix}
}

\newcommand{\matcolthree}[3]{%
\begin{pmatrix}
  #1\\#2\\#3
\end{pmatrix}
}

\newcommand{\mattwothree}[6]{%
\begin{pmatrix}
  #1 & #2 & #3\\ #4 & #5 & #6
\end{pmatrix}
}
\newcommand{\matthreetwo}[6]{%
	\begin{pmatrix}
		#1 & #2\\ #3 & #4\\ #5 & #6
	\end{pmatrix}
}

\newcommand{\matthree}[9]{%
\begin{pmatrix}
  #1 & #2 & #3\\ #4 & #5 & #6\\ #7 & #8 & #9
\end{pmatrix}
}

\newcommand{\matfour}[9]{%
  \def\argi{{#1}}%
  \def\argii{{#2}}%
  \def\argiii{{#3}}%
  \def\argiv{{#4}}%
  \def\argv{{#5}}%
  \def\argvi{{#6}}%
  \def\argvii{{#7}}%
  \def\argviii{{#8}}%
  \def\argix{{#9}}%
  \matfourRelay
}
\newcommand\matfourRelay[7]{%
\begin{pmatrix}
  \argi & \argii & \argiii & \argiv\\
  \argv & \argvi & \argvii & \argviii\\
  \argix & #1 & #2 & #3\\
   #4 & #5 & #6 & #7
\end{pmatrix}
}

\theoremstyle{plain}
\newtheorem{theorem}{Theorem}[section]

\newtheorem{prop}[theorem]{Proposition}
\newtheorem{lemma}[theorem]{Lemma}

\newtheorem{ques}[theorem]{Question}

\newtheorem{note}[theorem]{Note}

\makeatletter
\def\namedlabel#1#2{\begingroup
	\def\@currentlabel{#2}%
	\label{#1}\endgroup
}
\makeatother
\newtheorem*{thmA}{\bf{Theorem A}}

\newtheorem*{thmOmega}{\bf{Theorem} $\bm{\Gom}$}
\newtheorem*{thmSigma}{\bf{Theorem} $\bm{\Gs}$}
\newtheorem*{thmLambda}{\bf{Theorem} $\bm{\Gl}$}
\theoremstyle{definition}
\newtheorem{defn}[theorem]{Definition}

\theoremstyle{remark}
\newtheorem{remark}[theorem]{Remark}
\newtheorem{example}[theorem]{Example}
\numberwithin{equation}{section}
\begin{document}
\title[On the Surjectivity of Certain Maps III: The Unital Set Condition]{On the Surjectivity of Certain Maps III: The Unital Set Condition}
\author[C.P. Anil Kumar]{C.P. Anil Kumar}
\address{Room. No. 223, Middle Floor, Main Building, Harish-Chandra Research Institute, Chhatnag Road, Jhunsi, Prayagraj (Allahabad)-211019, Uttar Pradesh, INDIA}
\email{akcp1728@gmail.com}
\subjclass[2010]{Primary 13F05,13A15 Secondary 11D79,11B25,16U60}
\keywords{commutative rings with unity, generalized projective spaces associated to ideals}
\thanks{This work is done while the author is a Post Doctoral Fellow at Harish-Chandra Research Institute, Prayagraj(Allahabad), INDIA.}
\begin{abstract}
In this article, for generalized projective spaces with any weights, we prove four main theorems in three different contexts where the Unital Set Condition USC (Definition~\ref{defn:UnitalSetCond}) on ideals is further examined. In the first context we prove, in the first main Theorem~\ref{theorem:GenCRTSURJ}, the surjectivity of the Chinese remainder reduction map associated to the generalized projective space of an
ideal $\mcl{I}=\us{i=1}{\os{k}{\prod}}\mcl{I}_k$ with a given factorization into mutually co-maximal ideals $\mcl{I}_j,1\leq j\leq k$ where $\mcl{I}$ satisfies the USC, using the key concept of choice multiplier hypothesis (Definition~\ref{defn:CMH}) which is satisfied.
In the second context, for a positive $k$, we prove in the second main
Theorem~\ref{theorem:SurjModIdealSP}, the surjectivity of the reduction map $SP_{2k}(\R)\lra SP_{2k}(\frac{\R}{\I})$ of strong approximation type for a ring $\R$ quotiented by an ideal $\I$ which satisfies the USC.   
In the third context, for a positive integer $k$, we prove in the thrid main Theorem~\ref{theorem:FullGenSurj}, the surjectivity of the map from special linear group of degree $(k+1)$ to the product of generalized projective spaces of $(k+1)\operatorname{-}$mutually co-maximal
ideals $\mcl{I}_j,0\leq j\leq k$ associating the $(k+1)\operatorname{-}$rows or $(k+1)\operatorname{-}$columns, where the ideal $\mcl{I}=\us{j=0}{\os{k}{\prod}}\mcl{I}_j$ satisfies the USC. In the fourth main Theorem~\ref{theorem:FullGenSurjOne}, for a positive integer $k$, we prove the surjectivity of the map from the symplectic group of degree $2k$ to the product of generalized projective spaces of $(2k)\operatorname{-}$mutually co-maximal ideals $\mcl{I}_j,1\leq j\leq 2k$ associating the $(2k)\operatorname{-}$rows or $(2k)\operatorname{-}$columns where the ideal $\mcl{I}=\us{j=1}{\os{2k}{\prod}}\mcl{I}_j$ satisfies the USC. The answers to Questions~[\ref{ques:GenCRTSURJ},~\ref{ques:ProjHighDim},~\ref{ques:ProjHighDimOne}] 
in a greater generality are not known.
\end{abstract}
\maketitle

\section{\bf{Introduction}}

Generalized projective spaces are introduced in C.P. Anil Kumar~\cite{CPAKI} and surjectivity of certain maps are proved for special linear groups in C.P. Anil Kumar~\cite{CPAKI}. The surjectivity theorem of 
strong approximation type for special linear groups is proved in C.P. Anil Kumar~\cite{MR3887364}. Here in this article for a commutative ring $\R$ with unity, we prove surjectivity theorems for the infinite family of 
symplectic groups $SP_{2k}(\R),k\geq 1$ just similar to the infinite family special linear groups $SL_k(\R),k\geq 2$ in C.P.~Anil Kumar~\cite{MR3887364},~\cite{CPAKI}. We find and conclude that these two infinite 
families of groups behave similarly with respect to the surjectivity of certain maps, that is,
\begin{itemize}
	\item both the families give surjectivity of strong approximation type when the ring $\R$ is quotiented by an ideal $\I$ which satisfies the USC and
	\item both the families give surjective map onto the product of generalized projective spaces associated to mutually co-maximal ideals, when similar conditions are imposed on the ideals.
\end{itemize}
A survey of results on strong approximation can be found in~\cite{SA}. 
Another motivation to write this article is to answer the following two questions on surjectivity which are posed in C.~P.~Anil Kumar~\cite{CPAKI} in a considerable generality.

\begin{ques}
	\label{ques:GenCRTSURJ}
	Let $\R$ be a commutative ring with unity and $k,l\in \mbb{N}$. Let $\mcl{I}_i,1\leq i\leq k$ be mutually co-maximal ideals and
	$\mcl{I}=\us{i=1}{\os{k}{\prod}}\mcl{I}_i$. Let $m_j\in \mbb{N},0\leq j\leq l$. When is the following Chinese remainder reduction map associated to the generalized projective spaces 
	\equ{\mbb{PF}^{l,(m_0,m_1,\ldots,m_l)}_{\mcl{I}} \lra \mbb{PF}^{l,(m_0,m_1,\ldots,m_l)}_{\mcl{I}_1} \times \mbb{PF}^{l,(m_0,m_1,\ldots,m_l)}_{\mcl{I}_2} \times \ldots \times \mbb{PF}^{l,(m_0,m_1,\ldots,m_l)}_{\mcl{I}_k}}
	surjective (and even bijective)? Or under what further general conditions
	\begin{itemize}
		\item on the ring $\R$,
		\item on the values $m_j,0\leq j\leq l$,
		\item on the co-maximal ideals $\mcl{I}_1,\ldots,\mcl{I}_k$,
	\end{itemize}
	is this map surjective?
\end{ques}

\begin{ques}
	\label{ques:ProjHighDim}
	Let $\R$ be a commutative ring with unity. Let $k \in \mbb{N}$ and $\mcl{I}_0,\mcl{I}_1,\ldots,\mcl{I}_k$ be mutually co-maximal ideals in $\R$. Let $m_j^i\in \mbb{N}, 0\leq i,j\leq k$. 
	When is the map 
	\equ{SL_{k+1}(\R) \lra \us{i=0}{\os{k}{\prod}}\mbb{PF}^{k,(m^i_0,m^i_1,\ldots,m^i_k)}_{\I_i}} given by
	\equa{&A_{(k+1)\times (k+1)}=[a_{i,j}]_{0\leq i,j\leq k} \lra\\
		&\big([a_{0,0}:a_{0,1}:\ldots: a_{0,k}],[a_{1,0}:a_{1,1}:\ldots: a_{1,k}],\ldots,[a_{k,0}:a_{k,1}:\ldots: a_{k,k}]\big)}
	surjective? Or under what further general conditions
	\begin{itemize}
		\item on the ring $\R$,
		\item on the values $m_j^i,0\leq i,j\leq k$,
		\item on the co-maximal ideals $\mcl{I}_0,\mcl{I}_1,\ldots,\mcl{I}_k$,
	\end{itemize}
	is this map surjective?
\end{ques}
Also in this article we answer the following analogous question for symplectic groups apart from proving Theorem~\ref{theorem:SurjModIdealSP} of strong approximation type for them.
\begin{ques}
		\label{ques:ProjHighDimOne}
	Let $\R$ be a commutative ring with unity. Let $k \in \mbb{N}$ and $\mcl{I}_i,1\leq i\leq 2k$ be mutually co-maximal ideals in $\R$. Let $m_j^i\in \mbb{N}, 1\leq i,j\leq 2k$. 
	Is the map 
	\equ{SP_{2k}(\R) \lra \us{i=1}{\os{2k}{\prod}}\mbb{PF}^{2k-1,(m^i_1,m^i_2,\ldots,m^i_{2k})}_{\I_i}} given by
	\equa{&A_{2k\times 2k}=[a_{i,j}]_{1\leq i,j\leq k} \lra\\
		&\big([a_{1,1}:a_{1,2}:\ldots: a_{1,2k}],[a_{2,1}:a_{2,2}:\ldots: a_{2,2k}],\ldots,[a_{2k,1}:a_{2k,2}:\ldots: a_{2k,2k}]\big)}
	surjective? Or under what further general conditions
	\begin{itemize}
		\item on the ring $\R$,
		\item on the values $m_j^i,1\leq i,j\leq k$,
		\item on the co-maximal ideals $\mcl{I}_1,\mcl{I}_2,\ldots,\mcl{I}_{2k}$,
	\end{itemize}
	is this map surjective?
\end{ques}

We answer these three questions~[\ref{ques:GenCRTSURJ},~\ref{ques:ProjHighDim},~\ref{ques:ProjHighDimOne}]  in the affirmative, for any given values $m_j\in\N,0\leq j\leq l, m^i_j\in \N,0\leq i,j\leq k, m^i_j\in\N,1\leq i,j\leq 2k$ of the projective spaces respectively in these three questions with no further conditions on the ring $\R$, when the product ideals $\us{i=1}{\os{k}{\prod}}\mcl{I}_i,\us{i=0}{\os{k}{\prod}}\mcl{I}_i,\us{i=1}{\os{2k}{\prod}}\mcl{I}_i$ respectively in these three questions satisfy the USC. 


\section{\bf{The main results}}
In this section we state the main results. Before stating the main results we need to introduce the definition of generalized projective spaces and one more standard definition.


\begin{defn}[Definition of a Projective Space Relation]
\label{defn:ProjSpaceRelation}
~\\
Let $\R$ be a commutative ring with unity. Let $k\in \mbb{N}$ and \equ{\mcl{GCD}_{k+1}(\R)=\{(a_0,a_1,a_2,\ldots,a_k)\in \R^{k+1}\mid \us{i=0}{\os{k}{\sum}}\langle a_i\rangle=\R\}.} Let $\I \subsetneq \R$ be an ideal
and $m_0,m_1,\ldots,m_k\in \mbb{N}$. Define an equivalence relation \equ{\sim^{k,(m_0,m_1,\ldots,m_k)}_{\I}} on $\mcl{GCD}_{k+1}(\R)$
as follows. For \equ{(a_0,a_1,a_2,\ldots,a_k),(b_0,b_1,b_2,\ldots,b_k) \in \mcl{GCD}_{k+1}(\R)} we say \equ{(a_0,a_1,a_2,\ldots,a_k)\sim^{\{k,(m_0,m_1,\ldots,m_k)\}}_{\I}(b_0,b_1,b_2,\ldots,b_k)} if there exists a \equ{\gl\in \R\text{ with }\ol{\gl}\in \bigg(\frac{\R}{\I}\bigg)^{*}} such that we have 
\equ{a_i \equiv \gl^{m_i}b_i \mod \I,0\leq i\leq k.}  
\end{defn}
 
We mention an example now.

\begin{example}
We take in this example $\R=\Z,\mcl{I}=\langle p\rangle $ where $p$ is a prime, $k=1,m_0=1,m_1=2$. Consider the equivalence relation $\sim^{1,(1,2)}_{\langle p\rangle}$. Here the equivalence classes are given as follows.
	\begin{itemize}
		\item One equivalence class: $E_0=\{(a,pb)\mid a,b\in \Z,p\nmid a,gcd(a,pb)=1 \}$ which modulo $p$ has size $(p-1)$.
		\item $(p-1)$ equivalence classes: For $1 \leq b\leq p-1$ we have $E_b=\{(\gl+pa,\gl^2b+pc)\mid a,b,c,\gl \in \Z,p\nmid \gl,gcd(\gl+pa,\gl^2b+pc)=1\}$ each of which modulo $p$ has size $p-1$.
		\item One equivalence class: $R_1=\{(pa,\gl^2+pb)\mid a,b,\gl\in \Z, p\nmid \gl,gcd(pa,\gl^2+pb)=1\}$ which modulo $p$ has size $\frac{p-1}{2}$.
		\item One equivalence class: Let $n\in \Z$ be a fixed quadratic non-residue modulo $p$. Then we have $O_n=\{(pa,\gl^2n+pb)\mid a,b,\gl\in \Z, p\nmid \gl, gcd(pa,\gl^2n+pb)=1\}$ which modulo $p$ has size $\frac{p-1}{2}$.
	\end{itemize}
	The total number of equivalence classes is $p+2$ which is one more than the number of elements in the usual projective space $\mbb{PF}^1_p$.
	A set of representatives for the equivalence classes  is given by
	\begin{itemize}
		\item $(1,b)\in E_b, 0\leq b\leq p-1$.
		\item $(0,1)\in R_1$
		\item $(0,n)\in O_n$ where $n$ is a chosen quadratic non-residue modulo $p$.
	\end{itemize}
\end{example}

\begin{example}
We take in this example $\R=\mbb{R},\mcl{I}=\langle 0\rangle,k\in \N,m_i=2,0\leq i\leq k$.
In this case the equivalence relation $\sim^{k,(2,2,\ldots,2)}_{\langle 0\rangle}$ is such that the equivalence classes are rays emanating from origin in $\mbb{R}^{k+1}$.	
\end{example}

\begin{example}
We take in this example $\R=\mbb{F}_p,p$ a prime, $\mcl{I}=\langle 0\rangle,k\in \N,m_i=2,0\leq i\leq k$. In this case the equivalence relation $\sim^{k,(2,2,\ldots,2)}_{\langle 0\rangle}$ is such that the equivalence classes are rays emanating from origin in $\mbb{F}_p^{k+1}$ where a ray is defined as follows. For $v=(a_0,a_1,\ldots,a_k)\in \mbb{F}_p^{k+1}\bs \{0\}$ the ray is given by $R_{v}=\{\gl^2(a_0,a_1,\ldots,a_k)\mid \gl\in \mbb{F}_p^{*}\}$. There are more equivalence classes than just the number of lines passing through origin.	
\end{example}

We define the generalized projective space associated to an ideal in a commutative ring with unity.
\begin{defn}
\label{defn:GenProjSpace}
Let $\R$ be a commutative ring with unity. Let $k\in \mbb{N}$ and \equ{\mcl{GCD}_{k+1}(\R)=\{(a_0,a_1,a_2,\ldots,a_k)\in \R^{k+1}\mid \us{i=0}{\os{k}{\sum}}\langle a_i\rangle=\R\}.}
Let $m_0,m_1,\ldots,m_k\in \mbb{N}$ and $\I \subsetneq \R$ be an ideal. Let $\sim^{k,(m_0,m_1,\ldots,m_k)}_{\I}$ denote the equivalence relation as in Definition~\ref{defn:ProjSpaceRelation}. Then we define 
\equ{\mbb{PF}^{k,(m_0,m_1,\ldots,m_k)}_{\I}\eqdef \frac{\mcl{GCD}_{k+1}(\R)}{\sim^{k,(m_0,m_1,\ldots,m_k)}_{\I}}.}
If $\mcl{I}=\R$ then let $\sim^{k,(m_0,m_1,\ldots,m_k)}_{\I}$ be the trivial equivalence relation on $\mcl{GCD}_{k+1}(\R)$ where any two elements are related.
We define \equ{\mbb{PF}^{k,(m_0,m_1,\ldots,m_k)}_{\I} \eqdef \frac{\mcl{GCD}_{k+1}(\R)}{\sim^{k,(m_0,m_1,\ldots,m_k)}_{\I}}}
a singleton set having single equivalence class.
\end{defn}

\begin{remark}
Let $\mbb{K}$ be a field. Then by choosing $\R=\mbb{K},\mcl{I}=\langle 0\rangle,k\in \N,m_i=1,0\leq i\leq k$ we get the projective space $\mbb{PF}^{k,(1,1,\ldots,1)}_{\langle 0\rangle}$ which is the usual $k$-dimensional projective space $\mbb{PF}^k_{\mbb{K}}$ over the field $\mbb{K}$.
\end{remark}

We define here, when a finite subset of a commutative ring $\R$ is a unital set.
\begin{defn}
	\label{defn:unitalset}
	Let $\R$ be a commutative ring with unity. Let $k\in \mbb{N}$. We say a finite subset 
	\equ{\{a_1,a_2,\ldots,a_k\}\subs \R} consisting of $k\operatorname{-}$elements (possibly with repetition) 
	is unital or a unital set if the ideal generated by the elements of the set is a unit ideal.
\end{defn}


Based on the previous definition, we make a relevant definition, the USC for an ideal.
\begin{defn}[Unital set condition for an ideal]
	\label{defn:UnitalSetCond} Let $\R$ be a commutative ring with unity. Let $k\in \mbb{N}$ and $\mcl{I} \sbnq \mcl{\R}$ be an ideal. 
	We say $\mcl{I}$ satisfies unital set condition $USC$ if for every unital set
	$\{a_1,a_2,\ldots,a_k\} \subs \R$ with $k \geq 2$, there exists an
	element $b \in \langle a_2,\ldots,a_k\rangle$ such that $a_1+b$ is a unit modulo
	$\mcl{I}$.
\end{defn}
\begin{example}
	In the ring $\Z$ any ideal $0\neq \mcl{I}\subsetneq \Z$ satisfies the USC using Lemma~\ref{lemma:FundLemma}. More generally, in a commutative ring $\R$ with unity, any ideal $\mcl{I}\sbnq \R$ which is contained in only a finitely many maximal ideals of $\R$, satisfies the USC using Proposition~\ref{prop:Unital}. So any non-zero ideal $\I$ in a Dedekind domain $\R$ satisfies the USC. Also see~\ref{Example:USC} for some exotic variety of examples..
\end{example}

\subsection{\bf{The first main theorem}}
The first main theorem is stated as follows:


\begin{thmA}
\namedlabel{theorem:GenCRTSURJ}{A}
Let $\R$ be a commutative ring with unity and $k,l\in \mbb{N}$. Let $\mcl{I}_i,1\leq i\leq k$ be mutually co-maximal ideals 
and either the ideal $\mcl{I}=\us{i=1}{\os{k}{\prod}}\mcl{I}_i$ satisfies the USC or $\mcl{I}=\R$. Let $m_j\in \mbb{N},0\leq j\leq l$. Then the Chinese remainder reduction map associated to the projective space 
\equ{\mbb{PF}^{l,(m_0,m_1,\ldots,m_l)}_{\mcl{I}} \lra \mbb{PF}^{l,(m_0,m_1,\ldots,m_l)}_{\mcl{I}_1} \times \mbb{PF}^{l,(m_0,m_1,\ldots,m_l)}_{\mcl{I}_2} \times \ldots \times \mbb{PF}^{l,(m_0,m_1,\ldots,m_l)}_{\mcl{I}_k}}
is surjective (in fact bijective).
\end{thmA}

\begin{thmLambda}
	\namedlabel{theorem:SurjModIdealSP}{$\Gl$}
	Let $\R$ be a commutative ring with
	unity. Let $k\in \mbb{N}$. Let \equ{SP_{2k}(\R)=\bigg\{A\in M_{k\times k}(\R) \mid A^tJA=J\text{ where }J=\mattwo {0_{k\times k}}{I_{k\times k}}{-I_{k\times k}}{0_{k\times k}}\bigg\}.}
	Let $\I \sbnq \R$ be an ideal which satisfies the USC. Then the reduction map
	\equ{SP_{2k}(\R) \lra SP_{2k}\bigg(\frac{\R}{\I}\bigg)} is surjective.
\end{thmLambda}

\subsection{\bf{The third main theorem}}
The third main theorem is stated as follows:


\begin{thmOmega}
\namedlabel{theorem:FullGenSurj}{$\Gom$}
Let $\R$ be a commutative ring with unity. 
Let $k\in \mbb{N}$ and $\mcl{I}_i,0\leq i\leq k$ be mutually co-maximal ideals in $\R$ such that either the ideal $\mcl{I}=\us{i=0}{\os{k}{\prod}}\mcl{I}_i$ satisfies the USC or $\mcl{I}=\R$.
Let $m_j^i\in \mbb{N}, 0\leq i,j\leq k$. 
Then the map 
\equ{SL_{k+1}(\R) \lra \us{i=0}{\os{k}{\prod}}\mbb{PF}^{k,(m^i_0,m^i_1,\ldots,m^i_k)}_{\I_i}} given by
\equa{&A_{(k+1)\times (k+1)}=[a_{i,j}]_{0\leq i,j\leq k} \lra\\
&\big([a_{0,0}:a_{0,1}:\ldots: a_{0,k}],[a_{1,0}:a_{1,1}:\ldots: a_{1,k}],\ldots,[a_{k,0}:a_{k,1}:\ldots: a_{k,k}]\big)}
is surjective. 
\end{thmOmega}


\subsection{\bf{The fourth main theorem}}

The fourth main theorem is stated as follows:


\begin{thmSigma}
\namedlabel{theorem:FullGenSurjOne}{$\Gs$}
	Let $\R$ be a commutative ring with unity. 
	Let $k\in \mbb{N}$ and $\mcl{I}_i,1\leq i\leq 2k$ be mutually co-maximal ideals in $\R$ such that either the ideal $\mcl{I}=\us{i=1}{\os{2k}{\prod}}\mcl{I}_i$ satisfies the USC or $\mcl{I}=\R$.
	Let $m_j^i\in \mbb{N}, 1\leq i,j\leq 2k$. 
	Then the map 
	\equ{SP_{2k}(\R) \lra \us{i=1}{\os{2k}{\prod}}\mbb{PF}^{2k-1,(m^i_1,m^i_2,\ldots,m^i_{2k})}_{\I_i}} given by
	\equa{&A_{2k\times 2k}=[a_{i,j}]_{1\leq i,j\leq 2k} \lra\\
		&\big([a_{1,1}:a_{1,2}:\ldots: a_{1,2k}],[a_{2,1}:a_{2,2}:\ldots: a_{2,2k}],\ldots,[a_{2k,1}:a_{2k,2}:\ldots: a_{2k,2k}]\big)}
	is surjective. 
\end{thmSigma}
We give appropriate motivating examples which are very interesting and relevant for this aricle.  
\begin{example}
	Let $\R=\Z$. Let $k=3$ and $\mcl{I}_0=241\Z,\mcl{I}_1=601\Z,\mcl{I}_2=1201\Z,\mcl{I}_3=1321\Z$ be the four distinct primes ideals in $\Z$.
	Let \equ{M=[m^i_j]_{0\leq i,j\leq 3}=\begin{pmatrix}
			2 & 5 & 3 & 10\\
			8 & 20 & 30 & 24\\
			1 & 50 & 48 & 40\\
			11 & 55 & 44 & 22\\
		\end{pmatrix}.}
	
	Consider the row unital matrix say for example 
	\equ{B=[b_{ij}]_{0\leq i,j\leq 3}=\begin{pmatrix}
			1 & 1 & 1 & 1\\
			1 & 1 & 1 & 1\\
			1 & 1 & 1 & 1\\
			1 & 1 & 1 & 1\\
	\end{pmatrix}}
	which gives rise to the element  
	\equa{&([1:1:1:1],[1:1:1:1],[1:1:1:1],[1:1:1:1],[1:1:1:1])\in\\ &\mbb{PF}_{(241)}^{4,(2,5,3,10)}\times \mbb{PF}_{(601)}^{4,(8,20,30,24)}\times \mbb{PF}_{(1201)}^{4,(1,50,48,40)}\times \mbb{PF}_{(1321)}^{4,(11,55,44,22)}.}
	
	So does there exist a matrix $A=[a_{ij}]_{0\leq i,j\leq 3}\in SL_4(\Z)$ such that $a_{ij}\equiv \gl_i^{m^i_j}b_{ij} \mod \mcl{I}_i, 0\leq i,j\leq 3,\gl_i\in \Z$ such that $\ol{\gl_i}\in \frac{\Z}{\mcl{I}_i}$ is a unit? The answer is yes and there does exist such a matrix $A\in SL_4(\Z)$ using Theorem~\ref{theorem:FullGenSurj}.

	Also does there exist a matrix $A=[a_{ij}]_{0\leq i,j\leq 3}\in SP_4(\Z)$ such that $a_{ij}\equiv \gl_i^{m^i_j}b_{ij} \mod \mcl{I}_i, 0\leq i,j\leq 3,\gl_i\in \Z$ such that $\ol{\gl_i}\in \frac{\Z}{\mcl{I}_i}$ is a unit? The answer is yes and there does exist such a matrix $A\in SP_4(\Z)$ using Theorem~\ref{theorem:FullGenSurjOne}.

\end{example}

Now we mention an example where Theorem~\ref{theorem:FullGenSurj} need not hold if we just have mutual co-maximalty of the ideals $\mcl{I}_i,0\leq i\leq k$.
\begin{example}
	\label{Example:NotSurj}
	Consider the ring $\R=\frac{\mbb{R}[X_1,X_2,X_3]}{\langle X_1^2+X_2^2+X_3^2-1\rangle}=\mbb{R}[x_1=\ol{X_1},x_2=\ol{X_2},x_3=\ol{X_3}]$. This ring does not satisfy unimodular extension property. Note the vector $(x_1,x_2,x_3)$ is unimodular. However there does not exist a matrix in $GL_3(\R)$ such that the first row is $(x_1,x_2,x_3)$. This is because, there is no non-vanishing tangent vector field over the two dimensional sphere $S^2$. Let $\mcl{I}_0=\langle 0\rangle,\mcl{I}_1=\R,\mcl{I}_2=\R, m^i_j=1, 0\leq i,j\leq 2,k=2$. The map $SL_3(\R) \lra \mbb{PF}_{\mcl{I}_0}^{2,(1,1,1)} \times \mbb{PF}_{\mcl{I}_1}^{2,(1,1,1)} \times \mbb{PF}_{\mcl{I}_2}^{2,(1,1,1)}$ is not surjective as the element $([x_1:x_2:x_3],*,*)\in \mbb{PF}_{\mcl{I}_0}^{2,(1,1,1)} \times \mbb{PF}_{\mcl{I}_1}^{2,(1,1,1)} \times \mbb{PF}_{\mcl{I}_2}^{2,(1,1,1)}$ is not in the image. 	
\end{example}

We mention an example which answers Example $2.12$ in C.~P.~Anil Kumar~\cite{CPAKI}.
\begin{example}
	\label{Example:OQ}
Let $\mcl{R}=\mbb{R}[x,y]$. Let $k=2,r,s,t\in \N$. Let $\mcl{I}_0=\langle x,y\rangle^t, \mcl{I}_1=\langle x-1,y\rangle^r,\mcl{I}_2=\langle x,y-1\rangle^s$.
$M=[m^i_j]_{0\leq i,j\leq 2}=\matthree{2}{2}{2}{2}{2}{2}{2}{2}{2}$. Here Theorem~\ref{theorem:FullGenSurj} implies that the map $SL_3(\R)\lra \mbb{PF}_{\mcl{I}_0}^{2,(2,2,2)}\times  \mbb{PF}_{\mcl{I}_1}^{2,(2,2,2)} \times \mbb{PF}_{\mcl{I}_2}^{2,(2,2,2)}$ is surjective answering Example $2.12$ in C.~P.~Anil Kumar~\cite{CPAKI}.
\end{example}


\section{\bf{Definitions of some Classical Groups}}
Now we define various groups over a general commutative ring with unity.


\begin{defn}[Classical Groups over a Ring]
	Let $\R$ be a commutative ring with unity. Let $k,n,p,q\in \mbb{N}$. 
	\begin{enumerate}
		\item The $(2k\times 2k)\operatorname{-}$symplectic group over the ring $\R$ is defined as
		\equ{SP_{2k}(\R)=\bigg\{A\in M_{2k\times 2k}(\R) \mid A^tJA=J\text{ where }J=\mattwo {0_{k\times k}}{I_{k\times k}}{-I_{k\times k}}{0_{k\times k}}\bigg\}.}
		\item The $(n\times n)\operatorname{-}$orthogonal group over the ring $\R$ is defined as 
		\equ{O_{n}(\R)=\{A\in M_{n\times n}(\R) \mid A^tA=AA^t=I_{n\times n}\}.}
		\item The $(n\times n)\operatorname{-}$special orthogonal group over the ring $\R$ is defined as 
		\equ{SO_{n}(\R)=\{A\in M_{n\times n}(\R) \mid A^tA=AA^t=I_{n\times n},\Det(A)=1\}.}
		\item The $(p,q)\operatorname{-}$indefinite orthogonal group over the ring $\R$ is defined as 
		\equ{O_{p,q}(\R)=\bigg\{A\in M_{(p+q)\times (p+q)}(\R) \mid A^tJA=J\text{ where }J=\mattwo {I_{p\times p}}{0_{p\times q}}{0_{q\times p}}{-I_{q\times q}}\bigg\}.}
		\item The $(p,q)\operatorname{-}$special indefinite orthogonal group over the ring $\R$ is defined as 
		\equa{SO_{p,q}(\R)&=\bigg\{A\in M_{(p+q)\times (p+q)}(\R) \mid A^tJA=J,\Det(A)=1\text{ where }\\J&=\mattwo {I_{p\times p}}{0_{p\times q}}{0_{q\times p}}{-I_{q\times q}}\bigg\}.}
		\item The $(n\times n)\operatorname{-}$unipotent upper triangular group over the ring $\R$ is defined as 
		\equ{UU_n(\R)=\bigg\{A\in  M_{n\times n}(\R) \mid A \text{ is a unipotent upper triangular matrix}\bigg\}.}
	\end{enumerate}
\end{defn}


\section{\bf{Preliminaries}}
In this section we present some lemmas and propositions which are useful in the proof of main results.
\subsection{\bf{On arithmetic progressions}}
\label{sec:FundLemma}
In this section we prove a very useful lemma on arithmetic
progressions for integers and a proposition in the context of commutative rings with identity.
Remark~\ref{remark:FundLemma} below summarizes Lemma~\ref{lemma:FundLemma} and Proposition~\ref{prop:FundLemmaRings} in this section. 


\begin{lemma}[A lemma on Arithmetic Progressions for Integers]
\label{lemma:FundLemma}
~\\
Let $a,b\in \Z$ be integers such that $gcd(a,b)=1$. Let $m\in \Z$ be any non-zero integer. Then there exists $n_0\in \Z$ such that $gcd(a+n_0b,m)=1$.
\end{lemma}
\begin{proof}
Assume $a,b$ are both non-zero. Otherwise 
Lemma~\ref{lemma:FundLemma} is trivial. Let
$q_1,q_2,q_3,\ldots,q_t$ be the distinct prime factors of $m$.
Suppose $q \mid gcd(m,b)$ then $q \nmid a+nb$ for all $n \in
\mbb{Z}$. Such prime factors $q$ need not be considered. Let $q \mid
m, q \nmid b$. Then there exists $t_q \in \mbb{Z}$ such that the
exact set of elements in the given arithmetic progression divisible
by $q$ is given by \equ{\ldots, a+(t_q-2q)b, a+(t_q-q)b, a+t_qb,
a+(t_q+q)b,a+(t_q+2q)b \ldots} Since there are finitely many such
prime factors for $m$ which do not divide $b$ we get a set of
congruence conditions for the multiples of $b$ as $n \equiv t_q \mod\ q$. In order to get an $n_0$ we solve a different set of
congruence conditions for each such prime factor say for example $n
\equiv t_q+1 \mod\ q$. By Chinese remainder theorem we have such
solutions $n_0$ for $n$ which therefore satisfy $gcd(a+n_0b,m)=1$.
\end{proof}


\begin{prop}
\label{prop:FundLemmaRings} Let $\R$ be a commutative ring with identity. Let $f,g\in \R$ and they generate the unit ideal, that is, $\langle f\rangle+\langle g\rangle=\R$. Let $E$ be any
finite set maximal ideals in $\R$. Then there exists an element $a
\in \R$ such that $f+ag$ is a non-zero element in $\frac{\R}{\mcl{M}}$
for every $\mcl{M}\in E$.
\end{prop}
\begin{proof}
Let $E=\{\mcl{M}_1,\mcl{M}_2,\ldots,\mcl{M}_t\}$. If $g\in \mcl{M}_i$ then for all $a\in
\R,f+ag\nin \mcl{M}_i$. Otherwise both $f,g \in \mcl{M}_i$ which is a
contradiction to $1 \in \langle f,g\rangle$.

Consider the finitely many maximal ideals $\mcl{M}\in E$ such
that $g \nin \mcl{M}$. Then there exists $t_{\mcl{M}}$ such that the
set \equ{\{t \mid f+tg \in \mcl{M}\}=t_{\mcl{M}}+\mcl{M}} a complete
arithmetic progression. This can be proved as follows. Since 
$g \nin \mcl{M}$ we have $\langle g\rangle+\mcl{M}=\R$. So
there exists $t_{\mcl{M}}$ such that $f+t_{\mcl{M}}g \in \mcl{M}$.
If $f+tg \in \mcl{M}$ then $(t-t_{\mcl{M}})g \in \mcl{M}$. Hence
$t \in t_{\mcl{M}}+\mcl{M}$.

Since there are finitely many maximal ideals $\mcl{M}$ such that $g
\nin \mcl{M}$ in the set $E$ we get a finite set of congruence
conditions for the multiples $a$ of $g$ as $a \equiv t_{\mcl{M}}
\mod\ \mcl{M}$. In order to get an $a_0$ we solve a different set of
congruence conditions for each such maximal ideal in $E$ say for
example $a \equiv t_{\mcl{M}}+1 \mod\ \mcl{M}$. By Chinese Remainder
Theorem we have such solutions $a_0$ for $a$ which therefore satisfy
$f+a_0g \nin \mcl{M}$ for all maximal ideals $\mcl{M} \in E$ and
hence  $f+a_0g\nin \mcl{M}$ for every $\mcl{M} \in E$. This proves 
Proposition~\ref{prop:FundLemmaRings}.
\end{proof}


\begin{remark}
\label{remark:FundLemma}
If $a,b\in \Z,gcd(a,b)=1$ then there exist $x,y\in \Z$ such that $ax+by=1$.
Here we note that in general $x$ need not be one unless $a\equiv 1\mod\ b$.
However for any non-zero integer $m$ we can always choose $x=1$ to find
an integer $a+by$ such that $gcd(a+by,m)=1$. In the context of rings this observation
gives rise to elements which are outside a given finite set of maximal ideals.
\end{remark}


\subsection{\bf{The unital lemma}}
\label{sec:UnitalLemma}
In this section we prove unital Lemma~\ref{lemma:Unital} which is useful to obtain a
unit modulo a certain type of an ideal in a $k\operatorname{-}$row unital vector via an $SL_k(\R)\operatorname{-}$elementary
transformation. 


First we state a proposition which gives a criterion for the USC.
\begin{prop}
\label{prop:Unital} Let $\R$ be a commutative ring with unity. Let
$\mcl{J}\sbnq \R$ be an ideal contained in only a finitely many maximal
ideals. Then $\mcl{J}$ satisfies the USC, that is, if $k \geq 2$ is a positive integer and if $\{a_1,a_2,\ldots,a_k\} \subs
\R$ is a unital set i.e. $\us{i=1}{\os{k}{\sum}}\langle a_i\rangle =\mcl{R}$, then there exists $a \in \langle a_2,\ldots,a_k\rangle$ such that
$a_1+a$ is a unit mod $\mcl{J}$.
\end{prop}
\begin{proof}
Let $\{\mcl{M}_i:1\leq i\leq t\}$ be the finite set of maximal
ideals containing $\mcl{J}$. For example $\mcl{J}$ could be a
product of maximal ideals. Since the set $\{a_1,a_2,\ldots,a_k\}$ is
unital there exists $d \in \langle a_2,a_3,\ldots,a_k\rangle$ such that
$\langle a_1\rangle +\langle d\rangle=\langle 1\rangle$. We apply Proposition~\ref{prop:FundLemmaRings}, where $E=\{\mcl{M}_i:1\leq i\leq t\}$ to conclude that
there exists $n_0 \in \R$ such that $a=n_0d$ and $a_1+a=a_1+n_0d \nin
\mcl{M}_i$ for $1 \leq i\leq t$. This proves 
Proposition~\ref{prop:Unital}.
\end{proof}
\begin{example}
\label{Example:USC}
~\\
\begin{itemize}
	  \item	If $\R$ is a prinicipal ideal domain or more generally a Dedekind domain. Then any non-zero ideal $\I \sbnq \R$ satisfies the USC using Proposition~\ref{prop:Unital}. 
	
	\item 
	Now we give examples of rings $\mcl{S}$ and ideals $\mcl{J}\sbnq \R$ which satisfy the USC but $\mcl{J}$ is contained in infinitely many maximal ideals of $\R$. 
	\begin{enumerate}
	\item 
	For this let $\mathbb{K}$ be a field and $\mcl{S}=\us{i\in \N}{\prod} \mathbb{K}_i$, an infinite direct product of fields where $\mathbb{K}_i=\mathbb{K},i\in\N$. Let $\mcl{J}$ be the zero ideal. Then $\mcl{J}$ is contained in infinitely many maximal ideals and $\mcl{J}$ satisifes the USC.
	
	\item Here let $\mcl{O}_i,i\in \N$ be an infinite collection of Dedekind domains. Let $\mcl{J}_i\sbnq \mcl{O}_i$ be an infinite collection of non-zero ideals in each Dedekind domain $\mcl{O}_i,i\in \N$. Let $\mcl{S}=\us{i\in \N}{\prod}\mcl{O}_i,\mcl{J}=\us{i\in \N}{\prod}\mcl{J}_i$. Then $\mcl{J}$ is an ideal in the ring $\mcl{S}$ and $\frac{\mcl{S}}{\mcl{J}}\cong\us{i\in \N}{\prod} \frac{\mcl{O}_i}{\mcl{J}_i}$. Now the ideal $\mcl{J}$ is contained in infinitely many maximal ideals of $\mcl{S}$ even though $\mcl{J}_i$ is contained in only a finitely many maximal ideals of $\mcl{O}_i$ for each $i\in \N$. Since $\mcl{J}_i\sbnq \mcl{O}_i$ satisfies the USC for each $i\in \N$, we have $\mcl{J}\sbnq \mcl{O}$ satisfies the USC.
	\end{enumerate}

\end{itemize}	

\end{example}

\begin{lemma}
\label{lemma:Unital} Let $\R$ be a commutative ring with unity and
$k \geq 2$ be a positive integer. Let $\{a_1,a_2,\ldots,a_k\} \subs
\R$ be a unital set i.e. $\us{i=1}{\os{k}{\sum}}\langle a_i\rangle=\mcl{R}$ and
$E$ be a finite set of maximal ideals in $\R$. Then there exists
$a \in \langle a_2,\ldots,a_k\rangle$ such that $a_1+a \nin
\mcl{M}$ for all $\mcl{M} \in E$.
\end{lemma}
\begin{proof}
The proof is essentially similar to Proposition~\ref{prop:Unital} even though we need not have to construct an ideal
$\mcl{J}$ which is contained in exactly the maximal ideals in the set $E$.
\end{proof}

Now we prove two lemmas which are useful later.
\begin{lemma}
\label{lemma:UnitalSuperIdeal}
Let $\R$ be a commutative ring with unity and let $\mcl{I}\sbnq \R$ be an ideal which satisfies the USC. Let $\mcl{J}$ be any ideal in $\R$ such that $\mcl{I}\subseteq \mcl{J}\sbnq \mcl{R}$. Then $\mcl{J}$ also satisfies the USC.
\end{lemma}
\begin{proof}
The lemma follows immediately because any element $r\in \R$ which is a unit modulo $\mcl{I}$ is also a unit modulo $\mcl{J}$.	
\end{proof}
\begin{lemma}
\label{lemma:UnitalQuotientIdeal}
Let $\R$ be a commutative ring with unity and let $\mcl{I}\sbnq \R$ be an ideal which satisfies the USC. Let $k\geq 2$ and $\{a_1,a_2,\ldots,a_k\}$ be a unital set modulo $\mcl{I}$, that is, $\{\ol{a_1},\ldots,\ol{a_k}\}$ is a unital set in $\frac{\R}{\I}$. Then there exists $b\in \langle a_2,\ldots,a_k\rangle$ such that $a_1+b$ is a unit modulo $\I$. In other words if $\I$ satisfies the USC in $\R$ then the zero ideal satisfies the USC in $\frac{\R}{\I}$.
\end{lemma}
\begin{proof}
Let $\us{i=1}{\os{k}{\sum}}\ol{b_i}\ol{a_i}=1\in \frac{\R}{\I}$ for some $b_i\in \R$. Then there exists $t\in \I$ such that $\us{i=1}{\os{k}{\sum}}b_ia_i+t=1\in \R$ for some $t\in \I$. So the set $\{a_1,a_2,\ldots,a_k,t\}$ is a unital set in $\R$. Since $\I$ satisfies the USC, there exists $b'\in \langle a_2,\ldots,a_k,t\rangle$ such that $a_1+b'$ is a unit modulo $\I$. Now we have $b'=b+s$ for some $b\in \langle a_2,\ldots,a_k\rangle$ and $s\in \I$. Hence $a_1+b$ is a unit modulo $\I$. This proves the lemma.
\end{proof}

\subsection{\bf{Results of strong approximation type}}
\label{sec:SMR}
~\\
This is a result of strong approximation type.
Here we give a criterion called the USC which is
given in Definition~\ref{defn:UnitalSetCond} and mention the following surjectivity theorem which is stated as:
\begin{theorem}
\label{theorem:SurjModIdeal} Let $\R$ be a commutative ring with
unity. Let $k\in \mbb{N}$. Let \equ{SL_k(\R)=\{A\in M_{k\times k}(\R) \mid \Det(A)=1\}}
Let $\mcl{I} \sbnq \R$ be an ideal which satisfies the USC. Then the redsction map
\equ{SL_k(\R) \lra SL_k(\frac{\R}{\mcl{I}})} is surjective. 
\end{theorem}

A proof of this important Theorem~\ref{theorem:SurjModIdeal} can be found in~(refer to Theorem $1.7$ on Page $338$) C.~P.~Anil~Kumar~\cite{MR3887364}.  A survey of results on strong approximation can be found in A.~S.~Rapinchuk~\cite{SA}. 


\subsection{\bf{Choice multiplier hypothesis for a tuple with respect to an ideal}}
We define a key concept known as choice multiplier hypothesis $CMH$ for a tuple with respect to an ideal. This definition is useful in the proof of main Theorem~\ref{theorem:GenCRTSURJ}.
\begin{defn}[$CMH$ for a tuple with respect to an Ideal]
\label{defn:CMH}
~\\
Let $\R$ be a commutative ring with unity and $\mcl{I}\sbnq \R$ be
any ideal. Let $n > 1$ be any positive integer and
$(x_1,x_2,\ldots,x_n) \in \R^n$ be such that $\langle x_1\rangle+\langle x_2\rangle+\ldots+\langle x_n\rangle+\mcl{I}=\R$. Suppose $\R$ has the property
that there exist $a_1,a_2,\ldots,a_n \in \R$ such that
$\langle a_1\rangle +\langle a_2\rangle +\ldots+\langle a_n\rangle =\R$ and $a_1x_1+a_2x_2+\ldots+a_nx_n \in
1+\mcl{I}$. Then we say $\R$ satisfies $CMH$
for the tuple $(x_1,x_2,\ldots,x_n) \in \R^n$ with respect to the ideal $\mcl{I}$. 
\end{defn}
\begin{remark}
A class of rings which satisfy CMH with respect to any proper ideal is given in Theorem~\ref{theorem:CMH}.
\end{remark}
We prove two lemmas which are useful in the proof of main Theorem~\ref{theorem:GenCRTSURJ}.


\begin{lemma}
\label{lemma:CMH}
Let $\R$ be a commutative ring with unity. Let $\mcl{I} \sbnq \R$ be an ideal.
$\R$ always satisfies $CMH$ with respect to the ideal $\mcl{I}$ for any positive integer $n>1$ for all tuples $(x_1,x_2,\ldots,x_n)\in \R^n$
when one of the $x_i:1 \leq i \leq n$ is a unit $\mod \mcl{I}$.
\end{lemma}
\begin{proof}
Let $n\geq 2$. Let $(x_1,x_2,\ldots,x_n)\in \R^n$. Without loss of generality let $x_1$ be a unit modulo $\mcl{I}$. Let $ax_1+t=1$ for some $a_1\in \R,t\in \I$. 
Then we choose $a_1=a,a_2=t$ and $a_3=\ldots=a_n=0$. We have $\langle a_1\rangle +\langle a_2\rangle =\R$ and $a_1x_1+a_2x_2=ax+tx_2=1+t(x_2-1)\in 1+\I$. This proves the lemma.
\end{proof}
\begin{lemma}
\label{lemma:CMHimpliesUnitalVect}
Let $\R$ be a commutative ring with unity and $\I\sbnq R$ be an ideal. Let $n\in \mbb{N}_{>1}$ and
$(x_1,x_2,\ldots,x_n) \in \R^n$ be such that $\langle x_1\rangle +\langle x_2\rangle +\ldots+\langle x_n\rangle+\mcl{I}=\R$. Suppose $\R$ satisfies CMH for the tuple $(x_1,x_2,\ldots,x_n)\in \R^n$
with respect to the ideal $\I$. Then there exists $t_1,t_2,\ldots,t_n\in\I$ such that $\us{i=1}{\os{n}{\sum}}\langle x_i+t_i\rangle =\R$. 
\end{lemma}
\begin{proof}
By $CMH$ let $a_1,a_2,\ldots,a_n \in \R$ such that $\langle a_1\rangle +\langle a_2\rangle+\ldots+\langle a_n\rangle =\R$ and $a_1x_1+a_2x_2+\ldots+a_nx_n = 1-t \in 1+\mcl{I}$ where $t\in \I$.
Suppose $b_1a_1+b_2a_2+\ldots+b_na_n=1$. Then we have $a_1(x_1+tb_1)+a_2(x_2+tb_2)+\ldots+a_n(x_n+tb_n)=1$. Choosing $t_i=tb_i$ the lemma follows.
\end{proof}
\begin{theorem}
\label{theorem:CMH}
Let $\R$ be a Dedekind domain and $\mcl{I}\subsetneq \R$ be an ideal. Let $n\in \N_{>1}$ and $(x_1,x_2,\ldots,x_n) \in \R^n$ be such that $\langle x_1\rangle+\langle x_2\rangle+\ldots+\langle x_n\rangle +\mcl{I}=\R$. Then $\R$ satisfies $CMH$ for the tuple $(x_1,x_2,\ldots,x_n) \in \R^n$ with respect to the ideal $\mcl{I}$.
\end{theorem}
\begin{proof}
Using Proposition $2.27$ in C.~P.~Anil Kumar~\cite{MR3887364}, we obtain $t_1,t_2,\ldots,t_n\in \mcl{I}$ such that the set $\{x_1+t_1,x_2+t_2,\ldots,x_n+t_n\}$ is unital in $\R$. So there exists $a_1,a_2,\ldots,a_n\in \R$ such that $a_1(x_1+t_1)+a_2(x_2+t_2)+\ldots+a_n(x_n+t_n)=1$. Hence $\langle a_1\rangle +\langle a_2\rangle +\ldots+\langle a_n\rangle =\R$ and $a_1x_1+a_2x_2+\ldots+a_nx_n\in 1+\mcl{I}$. 
\end{proof}

\section{\bf{Chinese Remainder Theorem for Generalized Projective Spaces}}
In this section we prove the first main Theorem~\ref{theorem:GenCRTSURJ} which concerns the surjectivity of the Chinese remainder 
reduction map associated to a projective space of an ideal with a given co-maximal ideal factorization when the product ideal satisfies the USC.


\begin{proof}
If $\mcl{I}=\R$ then the proof is easy. Now we ignore unit ideals which occur in the factorization $\mcl{I}=\us{i=1}{\os{k}{\prod}}\mcl{I}_i$.
The theorem holds for $k=1$ and any $l\in \mbb{N}$ as the proof is immediate. We prove by induction on $k$. Let
\equa{([a_{10}:a_{11}:\ldots:a_{1l}],\ldots,&[a_{k0}:a_{k1}:\ldots:a_{kl}])
\in\\& \mbb{PF}^{l,(m_0,m_1,\ldots,m_l)}_{\mcl{I}_1} \times \mbb{PF}^{l,(m_0,m_1,\ldots,m_l)}_{\mcl{I}_2} \times
\ldots \times \mbb{PF}^{l,(m_0,m_1,\ldots,m_l)}_{\mcl{I}_k}.} By induction we
have an element $[b_0:b_1:b_2:\ldots:b_l] \in
\mbb{PF}^{l,(m_0,m_1,\ldots,m_l)}_{\mcl{I}_2\mcl{I}_3\ldots\mcl{I}_k}$ representing the
last $k-1$ elements. Consider the matrix \equ{A=\begin{pmatrix}
a_{10}  & a_{11}  & \cdots & a_{1,l-1} & a_{1l}\\
b_0& b_1& \cdots & b_{l-1} & b_l
\end{pmatrix}}
where we have $\us{i=0}{\os{l}{\sum}}\langle a_{1i}\rangle =\R=\us{i=0}{\os{l}{\sum}}\langle b_i\rangle$.
We change the matrix $A$ to a suitable matrix $B$ by applying elementary column operations on $A$ or right multiply $A$ by matrices in $SL_{l+1}(\R)$. We keep track of the matrices in $SL_{l+1}(\R)$ used for mulitplication for back tracking later by their inverses at the end. 
Now $\mcl{I}\subseteq \mcl{I}_1\sbnq \mcl{R}$. Hence $\mcl{I}_1$ satisfies the USC using Lemma~\ref{lemma:UnitalSuperIdeal}. By a suitable application of a lower triangular matrix in $SL_{l+1}(\R)$ with diagonal entries equal to one, we can assume $a_{10}$ is a unit modulo $\mcl{I}_1$. 
By finding inverse of this element modulo $\mcl{I}_1$ and hence again by a suitable application of $SL_{l+1}(\R)$ matrix, the matrix $A$ can
be transformed to the following matrix $B$, where $a_0$ in the first row is a unit modulo $\mcl{I}_1$ and $a_i\in \mcl{I}_1$ for $1\leq i\leq l$ and also $\us{i=0}{\os{l}{\sum}}\langle a_i\rangle =\R$.
\equ{B=\begin{pmatrix}
a_0  & a_1  & \cdots & a_{l-1} & a_l\\
c_0& c_1& \cdots & c_{l-1} & c_l
\end{pmatrix}}

Now $\mcl{I}\subseteq \mcl{I}_2\ldots \mcl{I}_k \sbnq \R$. Hence  $\mcl{I}_2\ldots \mcl{I}_k$ satisfies the USC using Lemma~\ref{lemma:UnitalSuperIdeal}.  If $c_0$ is not a unit $\mod \mcl{I}_2\ldots \mcl{I}_k$ then by a suitable application of a lower triangular matrix in $SL_{l+1}(\R)$ with diagonal entries equal to one, we can assume the first element $c_0$ in the second row of $B$ is a unit modulo $\mcl{I}_2\ldots\mcl{I}_k$
and the first element $a_0$ in the first row will still remain a unit modulo $\mcl{I}_1$ after the transformation.
We have the following facts on the matrix $B$ now.

\begin{enumerate}
\item $a_0$ is a unit modulo $\mcl{I}_1$. 
\item $a_i \in \mcl{I}_1$ for $1\leq i\leq l$.
\item $\us{i=0}{\os{l}{\sum}}\langle a_i\rangle =\R$.
\item $c_0$ is a unit modulo $\mcl{I}_2\ldots\mcl{I}_k$.
\item $\us{i=0}{\os{l}{\sum}}\langle c_i\rangle =\R$.
\end{enumerate}

By usual Chinese Remainder Theorem let $x_0\equiv a_0\mod \mcl{I}_1,x_0\equiv c_0 \mod \mcl{I}_2\ldots\mcl{I}_k$. For $1\leq i\leq l$ let $x_i\in \mcl{I}_1,x_i\equiv c_i\mod \mcl{I}_2\ldots\mcl{I}_k$.
Then we have $\langle x_0\rangle +\langle x_1\rangle +\ldots+\langle x_l\rangle +\mcl{I}_1=\R,\langle x_0\rangle +\langle x_1\rangle +\ldots+\langle x_l\rangle +\mcl{I}_2\ldots\mcl{I}_k=\R$ which implies $\langle x_0\rangle +\langle x_1\rangle +\ldots+\langle x_l\rangle +\mcl{I}=\R$. Moreover $x_0$ is a unit modulo $\I$ as it is a unit modulo
both $\mcl{I}_1$ and $\mcl{I}_2\ldots\mcl{I}_k$. Since $l\geq 1$ by using Lemmas~\ref{lemma:CMH},~\ref{lemma:CMHimpliesUnitalVect} (as $CMH$ is satisfied), there exist $t_0,t_1,\ldots,t_l \in \I$
such that $\us{i=0}{\os{l}{\sum}} \langle x_i+t_i\rangle =\R$ and a required element for $B$ is given by $[x_0+t_0:x_1+t_1:\ldots:x_l+t_l]\in \mbb{PF}^{l,(m_0,m_1,\ldots,m_l)}_{\mcl{I}}$.  By back tracking we can get a required element for $A$. 
Hence the induction step is completed thereby proving surjectivity.

Now we prove injectivity of the map which is slightly easier. Suppose $[a_0:a_1:\ldots:a_l],[b_0:b_1:\ldots:b_l]$ have the same image. Then there exist $\gl_i\in \R,1\leq i\leq k$ such that $\ol{\gl_i}\in \frac{\R}{\mcl{I}_i}$ is a unit and we have $a_j=\gl_i^{m_j}b_j \mod \mcl{I}_i, 0\leq j\leq l, 1\leq i\leq k$. Now we use usual chinese remainder theorem to obtain an element $\gl\in \R$ such that $\gl\equiv \gl_i\mod \mcl{I}_i$. Hence it follows that $\ol{\gl}\in \frac{\R}{\mcl{I}}$ is a unit and $a_j=\gl^{m_j}b_j \mod \mcl{I}, 0\leq j\leq l$. So we get $[a_0:a_1:\ldots:a_l]=[b_0:b_1:\ldots:b_l]$. 
Now Theorem~\ref{theorem:GenCRTSURJ} follows.
\end{proof}

\section{\bf{Surjectivity of the map $SP_{2k}(\R) \lra SP_{2k}(\frac{\R}{\I})$}}
First we prove a proposition which states that if there is a unit in a row or column then the row or column can be extended to a symplectic matrix.
\begin{prop}[Unit in a Row or Column Extension]
	\label{prop:UnitRowExt}
	~\\
	Let $\R$ be a commutative ring with unity and $k\in \N$. Let $(a_1,a_2,\ldots,a_k,a_{k+1},\ldots,a_{2k})$ $\in \R^{2k}$ be such that one of $a_i,1\leq i\leq 2k$ is a unit in $\R$. Then there exist matrices $g_1,g_2\in SP_{2k}(\R)$ such that 
	the $l^{th}$-row of $g_1$ is given by $(g_1)_{lj}=a_j,1\leq j\leq 2k$ and the $l^{th}$-column of $(g_2)_{jl}=a_j,1\leq j\leq 2k$ for any $1\leq i\leq 2k$. 
\end{prop}
\begin{proof}
	We suppose $l=1$ using a permutation matrix in $N\in SP_{2k}(\R)$ of the form either $N=\mattwo{P}{0}{0}{P}$ or $N=\mattwo{0}{P}{-P}{0}$ with $Ne^{2k}_{l}=e_1^{2k}$ without loss of generality.
	The first column of a matrix $M$ is the $l^{th}$-column of the matrix $MN$. 
	
	If $k=1$ then the proof is easy as $SP_2(\R)=SL_2(\R)$. So assume $k\geq 2$.
	We suppose first that $a_1$ is a unit. Then we consider the matrix
	\equa{A&=\Diag(a_1,a_1^{-1},1,1,\ldots,1)+\us{j=2}{\os{k}{\sum}}a_jE_{1j}\in SL_k(\R)\text{ with first row }(a_1,a_2,\ldots,a_k)\\
		&=\mattwo{\mattwo {a_1}{a_2}{0}{a_1^{-1}}}{\mattwothree{a_3}{\cdots}{a_k}{0}{\cdots}{0}}{0_{(k-2)\times 2}}{I_{(k-2)\times (k-2)}}.}
	Consider the symmetric matrix \equa{B&=\mattwo{a_1^{-1}a_{k+1}-\us{j=2}{\os{k}{\sum}}a_1^{-2}a_ja_{k+j}}{R_{1\times (k-1)}}{C_{(k-1)\times 1}}{0_{(k-1)\times (k-1)}} \text{ where}\\  
		R_{1\times (k-1)}&=(a_1^{-1}a_{k+2},a_1^{-1}a_{k+3},\ldots,a_1^{-1}a_{2k}),\\ C_{(k-1)\times 1}&=(a_1^{-1}a_{k+2},a_1^{-1}a_{k+3},\ldots,a_1^{-1}a_{2k})^t.} Then the first row of 
	$AB$ is given by $(a_{k+1},a_{k+2},\ldots,a_{2k})$. Now the matrix 
	\equ{g_1=\mattwo {A}{0}{0}{(A^t)^{-1}}\mattwo{I}{B}{0}{I}=\mattwo{A}{AB}{0}{(A^t)^{-1}}\in SP_{2k}(\R)}
	has the required first row. For obtaining $g_2$ we take $g_2=g_1^t=\mattwo{A^t}{0}{B^tA^t}{A^{-1}}$ $=\mattwo{A^t}{0}{BA^t}{A^{-1}}$. 
	Then $g_2$ has the required first column. 
	
	Also the first row of $g_2$ is given by $(a_1,0,0,\ldots,0)=a_1e^{2k}_1\in \R^{2k}$. Moreover the entries of the matrix $BA^t$ are in the ideal $\langle a_{k+1},\ldots,a_{2k}\rangle \subseteq \R$. This little observation is useful later in the proof of Theorem~\ref{theorem:FullGenSurjOne}.

	If $a_i$ is a unit for some $1<i\leq k$ then using in addition a permutation matrix $\mattwo {P}{0}{0}{P}\in SP_{2k}(\R)$ such that $Pe^k_i=e^k_1$ where $e^k_j=(0,\ldots,0,1,0\ldots,0)^t\in \R^k$
	the $j^{th}$-standard vector for $1\leq j\leq k$, the positions of $a_1$ and $a_i$ can be interchanged. Now we construct a matrix $g$ as before so that the first row is $g$ is $(a_i,a_2,\ldots,a_{i-1},a_1,a_{i+1},\ldots,a_{2k})$.
	Then we exchange back the positions of $a_i$ and $a_1$. Now we obtain a matrix $g_1 \in SP_{2k}(\R)$ such that the first row of $g_1$ is $(a_1,a_2,\ldots,a_k,a_{k+1},\ldots,a_{2k})$.
	If $a_{k+i}$ is a unit for some $1\leq i\leq k$ then using in addition a matrix of the form $\mattwo {0}{P}{-P}{0}\in SP_{2k}(\R)$ such that $Pe^k_i=e^k_1$ for a transposition permutation matrix $P$ corresponding to $1$ and $i$
	and changing the signs $a_{k+1},\ldots,a_{2k}$ we obtain a matrix $g\in SP_{2k}(\R)$ such that the first row of $g$ is $(-a_{k+i}$, $-a_{k+2},\ldots,-a_{k+i-1},-a_{k+1},-a_{k+i+1},\ldots,-a_{2k},a_{i},a_2,\ldots,a_{i-1},a_1,a_{i+1},\ldots,a_k)$. 
	Now we obtain a matrix $g_1 \in SP_{2k}(\R)$ such that the first row of $g_1$ is $(a_1,a_2,\ldots,a_k$, $a_{k+1},\ldots,a_{2k})$ again using the same matrix $\mattwo 0P{-P}0$ three times.
	The proof for the column case follows by taking transpose. This proves the proposition. 
\end{proof}

We prove the second main Theorem~\ref{theorem:SurjModIdealSP} below.
\begin{proof}[Proof of Theorem~\ref{theorem:SurjModIdealSP}]
	If $k=1$ then the theorem follows because $SP_2(\R)=SL_2(\R)$, $SP_2\bigg(\frac{\R}{\I}\bigg)$ $=SL_2\bigg(\frac{\R}{\I}\bigg)$. So assume $k\geq 2$.
	Let $G\subseteq SP_{2k}\bigg(\frac{\R}{\I}\bigg)$ be the image of the reduction map. We need to prove that $G=SP_{2k}\bigg(\frac{\R}{\I}\bigg)$.
	Let \equa{g&=\mattwo{A=[a_{i,j}]_{k\times k}}{B=[b_{i,j}]_{k\times k}}{C=[c_{i,j}]_{k\times k}}{D=[d_{i,j}]_{k\times k}}\in M_{2k}(\R)\text{ with }\\
		\ol{g}&=\mattwo{\ol{A}}{\ol{B}}{\ol{C}}{\ol{D}}\in SP_{2k}\bigg(\frac{\R}{\I}\bigg).}
	We keep simplifying the element $g$ in various steps by multiplying it with elements in $SP_{2k}(\R)$ and by changing sometimes the element $g$ by an element $h$ such that $\ol{g}=\ol{h}\in SP_{2k}(\frac{\R}{\I})$ till we obtain that $g\in SP_{2k}(\R)$ after simplification.
	\begin{center}
		\boxed{\text{Step I: First Simplification, $a_{1,1}$ is a unit modulo }\I}
	\end{center}
	We have $\ol{g}^tJ\ol{g}=J \Ra \Det(\ol{g})$ is a unit in $\frac{\R}{\I}$. So the first row \equ{(a_{1,1},\ldots,a_{1,k},b_{1,1},\ldots,b_{1,k})} is a unital vector modulo $\I$. Since $\I$ satisfies the USC, using Lemma~\ref{lemma:UnitalQuotientIdeal} 
	there exists \equ{(1,e_{1,2},\ldots,e_{1,k},f_{1,1},\ldots,f_{1,k})^t\in \R^{2k}} such that we have that
	\equ{t_1=a_{1,1}+\us{j=2}{\os{k}{\sum}}a_{1,j}e_{1,j}+\us{j=1}{\os{k}{\sum}}b_{1,j}f_{1,j}\text{ is a is unit modulo }\I.}
	Now using Proposition~\ref{prop:UnitRowExt} there exists a matrix $g_2=\mattwo{E}{0}{F}{(E^t)^{-1}}\in SP_{2k}(\R)$ such that the first column of $g_2$ is $(1,e_{1,2},\ldots,e_{1,k},f_{1,1},\ldots,f_{1,k})^t$.
	Here \equ{E=\mattwo{\mattwo{1}{0}{e_{1,2}}{1}}{0_{2\times (k-2)}}{\matthreetwo{e_{1,3}}{0}{\vdots}{\vdots}{e_{1,k}}{0}}{I_{(k-2)\times (k-2)}}.}
	Consider the matrix $gg_2$. The element $(gg_2)_{1,1}=t_1$ is a unit modulo $\I$. 
	\begin{center}
		\boxed{\text{Step II:Second Simplification $a_{1,1}=1$}}
	\end{center}
	Let \equa{g&=\mattwo{A=[a_{i,j}]_{k\times k}}{B=[b_{i,j}]_{k\times k}}{C=[c_{i,j}]_{k\times k}}{D=[d_{i,j}]_{k\times k}}\in M_{2k}(\R)\text{ with }\\
		\ol{g}&=\mattwo{\ol{A}}{\ol{B}}{\ol{C}}{\ol{D}}\in SP_{2k}\bigg(\frac{\R}{\I}\bigg)\text{ and }a_{1,1}\text{ is a unit modulo }\I.}
	Consider the matrix $\Diag(\ol{a}^{-1}_{1,1},\ol{a}_{1,1},1,\ldots,1)\in SL_k\bigg(\frac{\R}{\I}\bigg)$. There exists $U\in SL_k(\R)$ such that 
	$\ol{U}=\Diag(\ol{a}^{-1}_{1,1},\ol{a}_{1,1},1,\ldots,1)$ because the reduction map $SL_k(\R)\lra SL_k\bigg(\frac{\R}{\I}\bigg)$ is surjective for 
	$k\in \mbb{N}$. Let $g_3=\mattwo{U}{0}{0}{(U^t)^{-1}}\in SP_{2k}(\R)$. We have $(gg_3)_{1,1}\equiv 1\mod \I$. So change it to $1$.
	\begin{center}
		\boxed{\text{Step III:Third Simplification $(a_{1,1},a_{1,2},\ldots,a_{1,k})=e_1^k=(a_{1,1},a_{2,1},\ldots,a_{k,1})$}}
	\end{center}
	Let \equa{g&=\mattwo{A=[a_{i,j}]_{k\times k}}{B=[b_{i,j}]_{k\times k}}{C=[c_{i,j}]_{k\times k}}{D=[d_{i,j}]_{k\times k}}\in M_{2k}(\R)\text{ with }\\
		\ol{g}&=\mattwo{\ol{A}}{\ol{B}}{\ol{C}}{\ol{D}}\in SP_{2k}\bigg(\frac{\R}{\I}\bigg)\text{ and }a_{1,1}=1.}
	Let $g_4=\mattwo{V}{0}{0}{(V^t)^{-1}}\in SP_{2k}(\R),g_5=\mattwo{W}{0}{0}{(W^t)^{-1}}\in SP_{2k}(\R)$ be such that 
	\equ{V=\mattwo{1}{-a_{1,2}\cdots-a_{1,k}}{0_{(k-1)\times 1}}{I_{(k-1)\times (k-1)}}\text{ and }W=\mattwo{1}{0}{\matcolthree{-a_{2,1}}{\vdots}{-a_{k,1}}}{I_{(k-1)\times(k-1)}}.}
	Then we have $(g_5gg_4)_{1,1}=1,(g_5gg_4)_{1,j}=0=(g_5gg_4)_{j,1},2\leq j\leq k$.
	\begin{center}
		\boxed{\text{Step IV:Fourth Simplification: The first row of $B$ is zero, }}
		\boxed{\text{The first column of $C$ is zero.}}
	\end{center}
	Let \equa{g&=\mattwo{A=[a_{i,j}]_{k\times k}}{B=[b_{i,j}]_{k\times k}}{C=[c_{i,j}]_{k\times k}}{D=[d_{i,j}]_{k\times k}}\in M_{2k}(\R)\text{ with }\\
		\ol{g}&=\mattwo{\ol{A}}{\ol{B}}{\ol{C}}{\ol{D}}\in SP_{2k}\bigg(\frac{\R}{\I}\bigg)\text{ and }a_{1,1}=1,a_{1,j}=a_{j,1}=0,2\leq j\leq k.} 
	Now right multiply $g$ by the following matrix $g_6\in SP_{2k}(\R)$ given by 
	\equ{g_6=\mattwo{I_{k\times k}}{\mattwo{-b_{1,1}}{-b_{1,2}\cdots-b_{1,k}}{\matcolthree{-b_{1,2}}{\vdots}{-b_{1,k}}}{0_{(k-1)\times(k-1)}}}{0_{k\times k}}{I_{k\times k}}.}
	Then left multiply by the following matrix $g_7$ given by 
	\equ{g_7=\mattwo{I_{k\times k}}{0_{k\times k}}{\mattwo{-c_{1,1}}{-c_{2,1}\cdots-c_{k,1}}{\matcolthree{-c_{2,1}}{\vdots}{-c_{k,1}}}{0_{(k-1)\times (k-1)}}}{I_{k\times k}}.}
	Then the matrix $g_7gg_6$ has the required properties.
	
	\begin{center}
		\boxed{\text{Step V: Fifth Simplification }}
		\boxed{\text{$a_{1,j}=0=a_{j,1},2\leq j\leq k,a_{2,j}=0=a_{j,2},3\leq j\leq k$}}
		\boxed{\text{$a_{1,1}=1,a_{2,2}$ is a unit modulo $\I$ and $a_{2,2}=1$ if $k\geq 3$}}
	\end{center}
	Let \equa{&g=\mattwo{A=[a_{i,j}]_{k\times k}}{B=[b_{i,j}]_{k\times k}}{C=[c_{i,j}]_{k\times k}}{D=[d_{i,j}]_{k\times k}}\in M_{2k}(\R)\\
		&\text{with }\ol{g}=\mattwo{\ol{A}}{\ol{B}}{\ol{C}}{\ol{D}}\in SP_{2k}\bigg(\frac{\R}{\I}\bigg)\text{ and }\\
		&a_{1,1}=1,a_{1,j}=a_{j,1}=0,2\leq j\leq k,b_{1,j}=0=c_{j,1},1\leq j\leq k.} 
	
	Now we repeat steps I,II,III for the second row and column of $A$ keeping the first row and column of $A$ intact. The second row of $g$ is given by $(0,a_{2,2},a_{2,3},\ldots,a_{2,k}$, $b_{2,1},b_{2,2},\ldots,b_{2,k})$ 
	which is unital modulo the ideal $\I$. Using Lemma~\ref{lemma:UnitalQuotientIdeal}, there exists $(0,1,e_{2,3},\ldots,e_{2,k}$, $f_{2,1},f_{2,2},\ldots,f_{2,k})\in \R^{2k}$ such that we have that
	\equ{t_2=a_{2,2}+\us{j=3}{\os{k}{\sum}}a_{2,j}e_{2,j}+\us{j=1}{\os{k}{\sum}}b_{2,j}f_{2,j}\text{ is a unit modulo }\I.}
	Now using Proposition~\ref{prop:UnitRowExt} there exists a matrix $g_8=\mattwo{E}{0_{k\times k}}{F}{(E^t)^{-1}}\in SP_{2k}(\R)$ such that the second column of $g_8$ is 
	\equ{(0,1,e_{2,3},\ldots,e_{2,k},f_{2,1},f_{2,2},\ldots,f_{2,k})^t.}
	Here in fact  $E=I_{2\times 2}$ if $k=2$ and if $k\geq 3$ then \equa{&E=\mattwo{\matthree{1}{0}{0}{0}{1}{0}{0}{e_{2,3}}{1}}{0_{3\times (k-3)}}{\matthree{0}{e_{2,4}}{0}{\vdots}{\vdots}{\vdots}{0}{e_{2,k}}{0}}{I_{(k-3)\times (k-3)}}\in SL_k(\R).\\
		& \text{If } F_1=\mattwo{\mattwo{0}{f_{2,1}}{f_{2,1}}{f_{2,2}-\us{j=3}{\os{k}{\sum}}e_{2,j}f_{2,j}}}{\mattwothree{0}{\cdots}{0}{f_{2,3}}{\cdots}{f_{2,k}}}{\matthreetwo{0}{f_{2,3}}{\vdots}{\vdots}{0}{f_{2,k}}}{0_{(k-2)\times(k-2)}}
		\text{ a symm. matrix},\\
		&\text{then }g_8=\mattwo{I_{k\times k}}{0_{k\times k}}{F_1}{I_{k\times k}}\mattwo{E}{0_{k\times k}}{0_{k\times k}}{(E^t)^{-1}}=\mattwo{E}{0_{k\times k}}{F_1E=F}{(E^t)^{-1}}.}
	The first column of $g_8$ is given by $(1,0,\ldots,0,0,f_{2,1},\ldots,0)^t=e_1^{2k}+f_{21}e^{2k}_{k+2}$. The second column of $g_8$ is given by $(0,1,e_{2,3},\ldots,e_{2,k},f_{2,1},f_{2,2},\ldots,f_{2,k})^t$.
	The first two rows of $g_8$ are given by $e_1^{2k},e_2^{2k}$. Now consider the matrix $gg_8$. The principal $2\times 2$ sub-matrix of $gg_8$ is given by \equ{\mattwo{1}{0}{a_{2,1}+b_{2,2}f_{2,1}}{t_2}.}
	We make $t_2=1$ if $k\geq 3$ by right multiplying $g$ with a matrix $\mattwo{X}{0}{0}{(X^t)^{-1}}$ which is congruent modulo $\I$ to
	\equ{\mattwo{\Diag(1,t_2^{-1},t_2,1,\ldots,1)}{0_{k\times k}}{0_{k\times k}}{\Diag(1,t_2,t_2^{-1},1,\ldots,1)}}
	where $t_2^{-1}\in \R,t_2^{-1}t_2\equiv 1\mod \I, X\in SL_k(\R)$.
	Otherwise we keep $t_2$ as it is. Now using elementary matrices in $SL_k(\R)$ and hence the corresponding matrices in $SP_{2k}(\R)$ we make the first two columns and first two rows other than 
	the diagonal elements of $A$ as zeroes modulo $\I$ and hence zeroes.
	
	\begin{center}
		\boxed{\text{Step VI: Sixth Simplification}}
		\boxed{\text{$a_{1,j}=0=a_{j,1},2\leq j\leq k,a_{2,j}=0=a_{j,2},3\leq j\leq k$}}
		\boxed{\text{$a_{1,1}=1,a_{2,2}$ is a unit modulo $\I$ and $a_{2,2}=1$ if $k\geq 3$}}
		\boxed{\text{$b_{1,j}=b_{2,j}=0=c_{j,1}=c_{j,2},1\leq j\leq k$}}
	\end{center}
	Let \equa{g&=\mattwo{A=[a_{i,j}]_{k\times k}}{B=[b_{i,j}]_{k\times k}}{C=[c_{i,j}]_{k\times k}}{D=[d_{i,j}]_{k\times k}}\in M_{2k}(\R)\text{ with }\\
		\ol{g}&=\mattwo{\ol{A}}{\ol{B}}{\ol{C}}{\ol{D}}\in SP_{2k}\bigg(\frac{\R}{\I}\bigg)\text{ such that}\\
		A&=\mattwo{\mattwo{1}{0}{0}{t_2}}{0_{2\times (k-2)}}{0_{(k-2)\times 2}}{\ti{A}_{(k-2)\times (k-2)}}.}
	Right multiply $g$ by 
	\equa{g_9&=\mattwo{I_{k\times k}}{\ti{B}}{0}{I_{k\times k}}\\&=
		\mattwo{I_{k\times k}}{\mattwo{\mattwo{-b_{1,1}}{-b_{1,2}}{-b_{1,2}}{-b_{2,2}t_2^{-1}}}{\mattwothree{-b_{1,3}}{\cdots}{-b_{1,k}}{-b_{2,3}t_2^{-1}}{\cdots}{-b_{2,k}t_2^{-1}}}
			{\matthreetwo{-b_{1,3}}{-b_{2,3}t_2^{-1}}{\vdots}{\vdots}{-b_{1,k}}{-b_{2,k}t_2^{-1}}}{0_{(k-2)\times(k-2)}}}{0}{I_{k\times k}}.}
	The matrix $\ti{B}$ is symmetric and more importantly we have \equ{b_{1,2}\equiv b_{2,1}t_2^{-1}\mod \I.}
	This follows because we have \equ{\mattwo{1}{0}{0}{t_2}\mattwo{b_{1,1}}{b_{2,1}}{b_{1,2}}{b_{2,2}}\equiv \mattwo{b_{1,1}}{b_{1,2}}{b_{2,1}}{b_{2,2}}\mattwo{1}{0}{0}{t_2}\mod \I.} This is because we have 
	\equ{AB^t\equiv BA^t \mod \I \text{ since }\ol{g}=\mattwo{\ol{A}}{\ol{B}}{\ol{C}}{\ol{D}}\in SP_{2k}\bigg(\frac{\R}{\I}\bigg).}
	The matrix $g_9\in SP_{2k}(\R)$. If we rewrite $gg_9=\mattwo{A}{B}{C}{D}$ then the first two row entries of $B$ are zero modulo $\I$ and hence can be replaced by zeroes. 
	We similarly now make the first two column entries of $C$ zeroes.
	\begin{center}
		\boxed{\text{Step VII: Seventh Simplification}}
		\boxed{\text{$A=[a_{i,j}]_{i,j=1,\ldots,k}$ is a diagonal matrix of the form}}
		\boxed{\text{$\Diag(1,1,1,\ldots,t)$ where $t$ is a unit modulo $\I$}}
	\end{center}
	We repeat the steps done for the second row till we reach the $k^{th}$-row and $k^{th}$-column of $A$ keeping the previous rows 
	and columns of $A$ intact after the simplification. 
	Then we get a matrix \equa{g&=\mattwo{\Diag(1,\ldots,1,t)}{B}{C}{D} \text{ such that }\\
		\ol{g}&=\mattwo{\Diag(1,\ldots,1,\ol{t})}{\ol{B}}{\ol{C}}{\ol{D}}\in SP_{2k}\bigg(\frac{\R}{\I}\bigg)}
	\begin{center}
		\boxed{\text{Step VIII: Eighth Simplification}} 
		\boxed{\text{$A=[a_{i,j}]_{i,j=1,\ldots,k}$ is a diagonal matrix of the form}}
		\boxed{\text{$\Diag(1,\ldots,1,t)$ where $t$ is a unit modulo $\I,B=0=C$}}
		\boxed{\text{$D=\Diag(1,\ldots,1,t^{-1})$ where $t^{-1}\in \R$ and $tt^{-1}\equiv 1\mod \I$}}
	\end{center}
	Here we make in a similar manner $B=0=C$. To make $B$ zero we multiply $g$ by $g_8$ where 
	\equa{g_8&=\mattwo{I}{B'}{0}{I}\\ &=
		\mattwo{I_{k\times k}}{\matfour {-b_{1,1}}{\cdots}{-b_{1,(k-1)}}{-b_{1,k}}{\vdots}{\ddots}{\vdots}{\vdots}{-b_{1,(k-1)}}{\cdots}{-b_{(k-1),(k-1)}}{-b_{(k-1),k}}{-b_{1,k}}{\cdots}{-b_{(k-1),k}}{-t^{-1}b_{k,k}}}
		{0_{k\times k}}{I_{k\times k}}\in SP_{2k}(\R).}
	Here the matrix $B'$ is symmetric.
	Now we observe that since \equ{\ol{g}=\mattwo{\Diag(1,\ldots,1,\ol{t})}{\ol{B}}{\ol{C}}{\ol{D}}\in SP_{2k}\bigg(\frac{\R}{\I}\bigg)} we have that $\Diag(1,\ldots,1,t)B^t\equiv B\Diag(1,\ldots,1,t)\mod \I$. Therefore 
	$b_{i,j}\equiv b_{j,i},1\leq i,j\leq k-1$ and $b_{i,k}\equiv t^{-1}b_{k,i}\mod \I, 1\leq i\leq k-1$. So the sub-matrix $B$ can be made zero. Similarly the matrix $C$ can be made zero.
	
	So the matrix $g$ reduces to 
	\equa{g&=\mattwo{\Diag(1,\ldots,1,t)}{0}{0}{D}\in M_{2k}(\R)\text{ with }\\
		\ol{g}&=\mattwo{\Diag(1,\ldots,1,\ol{t})}{0}{0}{\ol{D}}\in SP_{2k}\bigg(\frac{\R}{\I}\bigg)}
	So we get $\ol{D}=\Diag(1,\ldots,1,\ol{t}^{-1})$. Therefore $g=\Diag(1,\ldots,1,t,1,\ldots,1,t^{-1})$ where $t^{-1}\in \R$ and $t^{-1}t\equiv 1\mod \I$.
	We conclude that $\Det(\ol{g})=1\in \frac{\R}{\I}$ that is $SP_{2k}(\frac{\R}{\I})\subseteq SL_{2k}(\frac{\R}{\I})$. 
	Since $\ol{g}\in SL_{2k}(\frac{\R}{\I})$ we have that there exists a matrix $h\in SL_{2k}(\R)$ such that $\ol{h}=\ol{g}$.
	\begin{center}
		\boxed{\text{Step IX: Final Step}} 
	\end{center}
	Now consider the matrix $\mattwo {x}{y}{z}{w}\in SL_2(\R)$ such that $\mattwo {\ol{x}}{\ol{y}}{\ol{z}}{\ol{w}} \equiv \mattwo{\ol{t}}{0}{0}{\ol{t}^{-1}}\in SL_2\bigg(\frac{\R}{\I}\bigg)=SP_2\bigg(\frac{\R}{\I}\bigg)$.
	Then consider the following matrix 
	\equ{\mattwo{\Diag(1,\ldots,1,x)}{\Diag(0,\ldots,0,y)}{\Diag(0,\ldots,0,z)}{\Diag(1,\ldots,1,w)}.}
	This matrix is in $SP_{2k}(\R)$ and reduces to $\ol{g}$.
	This completes the proof of Theorem~\ref{theorem:SurjModIdealSP}.
\end{proof}


\section{\bf{On the Generalized Surjectivity Theorem for Generalized Projective Spaces For Special Linear Groups}}
In this section we prove two surjectivity theorems in the context of generalized projective spaces for special linear groups. More precisely we prove
following Theorem~\ref{theorem:GenSurjMainGenIdeals} and the third main Theorem~\ref{theorem:FullGenSurj}.

\begin{theorem}
\label{theorem:GenSurjMainGenIdeals}
Let $\R$ be a commutative ring with unity. 
Let $k\in \mbb{N}$ and $\mcl{I}_0,\mcl{I}_1,\ldots,\mcl{I}_k$ be $(k+1)$ co-maximal ideals in $\R$ such that either $\mcl{I}=\us{i=0}{\os{k}{\prod}}\mcl{I}_i$ satisfies the USC or $\mcl{I}=\R$. 
Let $A_{(k+1)\times (k+1)}=[a_{i,j}]_{0\leq i,j \leq k}\in M_{(k+1)\times (k+1)}(\R)$ be such that for every $0\leq i\leq k$ the $i^{th}\operatorname{-}$row is unital, that is, $\us{j=0}{\os{k}{\sum}} \langle a_{i,j}\rangle =\R$
for $0\leq i\leq k$. Then there exists $B=[a_{i,j}]_{0\leq i,j \leq k}\in SL_{k+1}(\R)$ such that we have 
$a_{i,j}\equiv b_{i,j}\mod \mcl{I}_i,0\leq i,j\leq k$.
\end{theorem}


Before proving these two theorems we need some results which we will state and prove. 


\begin{prop}
\label{prop:diagdetone}
Let $\R$ be a commutative ring with unity. Let $1<k\in \mbb{N}$ and $\mcl{I}_1,\mcl{I}_2,\ldots,\mcl{I}_k$ be
mutually co-maximal ideals in $\R$. Let $a_1,a_2,\ldots,a_k\in \R$ be such that for $1\leq i\leq k, a_i$ is a unit modulo $\mcl{I}_i$ if $\mcl{I}_i\neq \R$. 
Then there exist $d_i \equiv a_i \mod \mcl{I}_i, 1\leq i\leq k$ such that 
\equ{d_1d_2\ldots d_k \equiv 1 \mod \mcl{I}_1\mcl{I}_2\ldots \mcl{I}_k.} 
\end{prop}
\begin{proof}
If all $\mcl{I}_i$ are unit ideals then we choose $d_i=1,1\leq i\leq k$. If one of them (say) $\mcl{I}_1\neq\R$ and $\mcl{I}_2=\mcl{I}_3=\ldots =\mcl{I}_k=\R$ then, let $z_1\in \R$ be such that $z_1a_1\equiv 1\mod \mcl{I}_1$.
Choose $d_1=a_1,d_2=z_1,d_3=\ldots =d_k=1$. Now we can assume that there are at least two ideals (say) $\mcl{I}_1\neq \R\neq \mcl{I}_2$. Here we choose $d_j=1$ if $\mcl{I}_j=\R$ for some $j>2$.
Hence we ignore unit ideals and assume that none of the ideals are unit ideals and $k\geq 2$. Let $b_i\in \R,1\leq i\leq k$ such that $b_ia_i\equiv 1\mod \mcl{I}_i$. 
Now we solve the following congruence equations for $d_1,d_2,\ldots,d_k$ using chinese remainder theorem for pairwise co-maximal ideals $\mcl{I}_i,1\leq i\leq k$.
\equa{&d_1\equiv a_1\mod \mcl{I}_1,d_1\equiv b_2 \mod \mcl{I}_2,d_1\equiv b_3\mod \mcl{I}_3,\ldots, d_1\equiv b_k\mod \mcl{I}_k\\
&d_2\equiv b_1\mod \mcl{I}_1,d_2\equiv a_2\mod \mcl{I}_2,d_2\equiv 1\mod \mcl{I}_3,\ldots,d_2\equiv 1\mod \mcl{I}_k\\
&\text{ for }3\leq i\leq k, d_i\equiv a_i\mod \mcl{I}_i,d_i\equiv 1\mod \mcl{I}_j\text{ for }j\neq i,1\leq j\leq k.}
So we have  $d_1d_2\ldots d_k \equiv 1 \mod \mcl{I}_1\mcl{I}_2\ldots \mcl{I}_k$ because 
$d_1d_2\ldots d_k \equiv 1 \mod \mcl{I}_i,1\leq i\leq k$.
\end{proof}


Now we state and prove a useful proposition.


\begin{prop}
\label{prop:bringunit}
Let $\R$ be a commutative ring with unity. Let $k\in \mbb{N}\cup\{0\}$ and $(a_0,a_1,\ldots,a_k)\in \R^{k+1}$ be a unital vector that is $\us{i=0}{\os{k}{\sum}}\langle a_i\rangle =\R$. 
Let $\mcl{I} \sbnq \R, \mcl{J} \subs \R$ be pairwise co-maximal ideals and $\mcl{I}$ satisfies the USC. For any subscript $0\leq i\leq k$ there exist 
\equ{x_0,x_1,\ldots, x_{i-1}\in \mcl{J} \text{ and }x_{i+1},\ldots,x_k\in \mcl{J}} such that the element 
\equ{\us{j=0}{\os{i-1}{\sum}}x_ja_j+a_i+\us{j=i+1}{\os{k}{\sum}}x_ja_j}
is a unit modulo the ideal $\mcl{I}$. 
\end{prop}
\begin{proof}
Let $q_1\in \mcl{I},q_2\in \mcl{J}$ be such that $q_1+q_2=1$. Since $\mcl{I}$ satisfies the USC, there exists $y_1,y_2,\ldots,y_{i-1},y_{i+1},\ldots,y_n\in \R$ such that $\us{j=1}{\os{i-1}{\sum}}y_ja_j+a_i+\us{j=i+1}{\os{n}{\sum}}y_ja_j$ is a unit modulo $\mcl{I}$. Now choose $x_j=q_2y_j,j\neq i, 1\leq j\leq k$. Then we have that 
\equ{\us{j=1}{\os{i-1}{\sum}}x_ja_j+a_i+\us{j=i+1}{\os{n}{\sum}}x_ja_j\equiv \us{j=1}{\os{i-1}{\sum}}y_ja_j+a_i+\us{j=i+1}{\os{n}{\sum}}y_ja_j \mod \mcl{I}.} Hence $\us{j=1}{\os{i-1}{\sum}}x_ja_j+a_i+\us{j=i+1}{\os{n}{\sum}}x_ja_j$ is a unit modulo $\mcl{I}$. 
This completes the proof of the proposition.
\end{proof}


Now we prove Theorem~\ref{theorem:GenSurjMainGenIdeals}.
\begin{proof}
First we make some observations.
\begin{enumerate}
\item The $i^{th}\operatorname{-}$row of $A_{(k+1)\times (k+1)}$ is unital if and only if the $i^{th}\operatorname{-}$row of \linebreak $A_{(k+1)\times (k+1)}.C$ is unital for any 
$C\in SL_{k+1}(\R)$.
\item The conclusion holds for the matrix $A_{(k+1)\times (k+1)}$ if and only if the conclusion holds for the matrix $A_{(k+1)\times (k+1)}.C$ for any $C\in SL_{k+1}(\R)$.
\item For any $0\leq i\leq k$, we can replace the $i^{th}\operatorname{-}$row $[a_{i,0}\ \ldots \ a_{i,k}]$ of $A_{(k+1)\times (k+1)}$ by any another unital row $[\ti{a}_{i,0}\ \ldots \ \ti{a}_{i,k}]$ such that $a_{i,j}\equiv \ti{a}_{i,j}\mod \mcl{I}_i, 0\leq j\leq k$.  
\end{enumerate}
If $\mcl{I}=R$ then $\mcl{I}_i=\R,0\leq i\leq k$ we can simply take $B$ as the identity matrix.
So we can assume that $\mcl{I}\sbnq \R$ and satisfies the USC. So if for any $0\leq i\leq k, \mcl{I}\subseteq\mcl{I}_i\sbnq \R$ then using Lemma~\ref{lemma:UnitalSuperIdeal}, $\mcl{I}_i$ also satisfies the USC.

We prove this theorem in several steps with the central idea being to transform $A_{(k+1)\times (k+1)}$ to another matrix for which the conclusion of the theorem holds.

\begin{enumerate}[label=Step(\Alph*):]
\item If $\mcl{I}_0\neq \R$ then $\mcl{I}_0$ satisfies the USC and we can make $a_{0,0}$ a unit modulo $\mcl{I}_0$ by applying a lower triangular $SL_{k+1}(\R)$ matrix with diagonal entries equal to one.
Let $z_{0,0}\in \R$ such that $z_{0,0}a_{0,0} \equiv 1 \mod \mcl{I}_0$.
Now we add $-z_{0,0}a_{0,j}$ times the $0^{th}\operatorname{-}$column to the $j^{th}\operatorname{-}$column for $1\leq j\leq k$. The $0^{th}\operatorname{-}$row becomes $[a_{0,0}\ \ldots \ a_{0,k}]$ with the following
properties.
\begin{itemize}
\item $a_{0,0}$ is a unit modulo $\mcl{I}_0$ if $\mcl{I}_0\neq \R$.
\item $a_{0,j}\in \mcl{I}_0$ for $1\leq j\leq k$.
\item $\us{j=0}{\os{k}{\sum}}(a_{0,j})=\R$, that is the $0^{th}\operatorname{-}$row is a unital vector.
\end{itemize}
In the subsequent steps we preserve these properties of the $0^{th}\operatorname{-}$row.
\item We inductively consider the $i^{th}\operatorname{-}$row of $A_{(k+1)\times (k+1)}$ for $i>0$.
Here if $\mcl{I}_i\neq \R$ and $a_{i,i}$ is not a unit modulo $\mcl{I}_i$ then $\mcl{I}_i$ satisfies the USC and we use Proposition~\ref{prop:bringunit} for the subscript $i$ and for the ideals $\mcl{I}=\mcl{I}_i$ and $\mcl{J}=\mcl{I}_0\mcl{I}_1\ldots\mcl{I}_{i-1}$ which is co-maximal with $\mcl{I}_i$ to make $a_{i,i}$ a unit modulo $\mcl{I}_i$. In this procedure the matrix $A$ will have the following properties.
\begin{itemize}
\item $a_{l,l}$ is a unit modulo $\mcl{I}_l$ for $0\leq l\leq i$ if $\mcl{I}_l\neq \R$.
\item $a_{l,j}\in \mcl{I}_j$ for $j\neq l, 0\leq j\leq k, 0\leq l\leq i$ when Proposition~\ref{prop:bringunit} is applied approriately, (that is, for the $i^{th}\operatorname{-}$row the subscript
$i$ is chosen, the ideal $\mcl{I}$ is chosen as $\mcl{I}_i$ and the ideal $\mcl{J}$ is chosen as $\mcl{I}_0\mcl{I}_1\ldots\mcl{I}_{i-1}$.
\item All rows of $A_{(k+1)\times (k+1)}$ are unital. 
\end{itemize}
\item We continue this procedure till the last $k^{th}\operatorname{-}$row.
After this procedure we have the following properties for $A_{(k+1)\times (k+1)}$.
\begin{itemize}
\item The diagonal entry $a_{i,i}$ is a unit modulo $\mcl{I}_i$ if $\mcl{I}_i\neq \R$ for $0\leq i\leq k$.
\item The non-diagonal entry $a_{i,j}\in \mcl{I}_i$ for $0\leq i\neq j\leq k$.
\item All rows of $A_{(k+1)\times (k+1)}$ are unital. 
\end{itemize}
\item Consider only the diagonal part of the matrix $A_{(k+1)\times (k+1)}$, that is, the matrix $D_{(k+1)\times (k+1)}=diag(a_{0,0},a_{1,1},\ldots,a_{k,k})$. 
We use Proposition~\ref{prop:diagdetone} to change the diagonal matrix to $D_{(k+1)\times (k+1)}$ to $\ti{D}_{(k+1)\times (k+1)}=diag(d_0,d_1,$ $\ldots,d_k)$ such that we have 
$d_i\equiv a_{i,i}\mod \mcl{I}_i,0\leq i\leq k$ and $d_0d_1\ldots d_k\equiv 1\mod \mcl{I}_0\mcl{I}_1\ldots\mcl{I}_k$. The ideal $\mcl{I}_0\mcl{I}_1\ldots\mcl{I}_k \neq \R$ by assumption and satisfies the USC. Hence 
\equ{\ol{\ti{D}}_{(k+1)\times (k+1)} \in SL_{k+1}\bigg(\frac{\R}{\mcl{I}_0\mcl{I}_1\ldots \mcl{I}_k}\bigg).}
and by an application of Theorem~\ref{theorem:SurjModIdeal} we conclude that the reduction map 
\equ{SL_{k+1}(\R) \lra SL_{k+1}\bigg(\frac{\R}{\mcl{I}_0\mcl{I}_1\ldots \mcl{I}_k}\bigg)} is surjective. Therefore there exists a matrix $B=[b_{i,j}]_{0\leq i,j\leq k}\in SL_{k+1}(\R)$ such that we have 
\equ{\ol{B}=\ol{\ti{D}}_{(k+1)\times (k+1)} \in SL_{k+1}\bigg(\frac{\R}{\mcl{I}_0\mcl{I}_1\ldots \mcl{I}_k}\bigg).}
\item This matrix $B$ is a required matrix. We observe the following. 
\begin{itemize}
\item The diagonal entries $b_{i,i}\equiv d_i\equiv a_{i,i} \mod \mcl{I}_i$.
\item The non-diagonal entries $b_{i,j}\in \mcl{I}_i$ and hence $b_{i,j}\equiv a_{i,j}$ for $0\leq i\neq j\leq k$.
\end{itemize}
\end{enumerate}
We have completed the proof of the theorem in Steps $(A)-(E)$.
\end{proof}


Now we prove the third main Theorem~\ref{theorem:FullGenSurj}.
\begin{proof}
Let \equa{\big([a_{0,0}:a_{0,1}:\ldots: a_{0,k}],[a_{1,0}:a_{1,1}:\ldots: a_{1,k}],\ldots,&[a_{k,0}:a_{k,1}:\ldots: a_{k,k}]\big)\\&\in \us{i=0}{\os{k}{\prod}}\mbb{PF}^{k,(m^i_0,m^i_1,\ldots,m^i_k)}_{\I_i}.}
Consider $A_{(k+1)\times (k+1)}=[a_{i,j}]_{0\leq i,j\leq k}\in M_{(k+1)\times (k+1)}(\R)$ for which Theorem~\ref{theorem:GenSurjMainGenIdeals} can be applied. Therefore we get 
$B=[b_{i,j}]_{0\leq i,j\leq k} \in SL_{k+1}(\R)$ such that $b_{i,j}\equiv a_{i,j} \mod \mcl{I}_i,0\leq i,j\leq k$. Hence we get 
\equ{[b_{i,0}:b_{i,1}:\ldots: b_{i,k}]=[a_{i,0}:a_{i,1}:\ldots: a_{i,k}]\in \mbb{PF}^{k,(m^i_0,m^i_1,\ldots,m^i_k)}_{\I_i},0\leq i\leq k.}
This proves the third main Theorem~\ref{theorem:FullGenSurj}.
\end{proof}

\begin{remark}
Having proved Theorem~\ref{theorem:FullGenSurj}, we have generalized Theorem $1.8$ on Page $339$ of C.~P.~Anil~Kumar~\cite{MR3887364} for ordinary projective spaces by removing the Dedekind type domain condition on the ring.
On the other hand we have no conditions on the ring in Theorem~\ref{theorem:FullGenSurj}, in the context of all generalized projective spaces, except that the product ideal $\mcl{I}$ of the co-maximal ideals $\mcl{I}_i,0\leq i\leq k$ must satisfy the USC, if $\mcl{I}$ is a proper ideal. 
\end{remark}


\section{\bf{On the Generalized Surjectivity Theorem for Generalized Projective Spaces For Symplectic Linear Groups}}
In this section we prove two surjectivity theorems in the context of generalized projective spaces for symplectic linear groups. More precisely we prove
Theorem~\ref{theorem:GenSurjMain} and fourth main Theorem~\ref{theorem:FullGenSurjOne}.

\begin{theorem}
	\label{theorem:GenSurjMain}
	Let $\R$ be a commutative ring with unity. 
	Let $k\in \mbb{N}$ and $\mcl{I}_1,\mcl{I}_1,\ldots,\mcl{I}_{2k}$ be $2k$ co-maximal ideals in $\R$ such that either $\mcl{I}=\us{i=1}{\os{2k}{\prod}}\mcl{I}_i$ satisfies the USC or $\mcl{I}=\R$.
	Let $M_{(2k)\times (2k)}=[m_{i,j}]_{1\leq i,j \leq 2k}\in M_{(2k)\times (2k)}(\R)$ such that for every $1\leq i\leq 2k$ the $i^{th}\operatorname{-}$row is unital, that is, $\us{j=1}{\os{2k}{\sum}} \langle m_{i,j}\rangle =\R$
	for $1\leq i\leq 2k$. Then there exists $N=[n_{i,j}]_{1\leq i,j \leq 2k}\in SP_{2k}(\R)$ such that we have 
	$m_{i,j}\equiv n_{i,j}\mod \mcl{I}_i,1\leq i,j\leq 2k$.
\end{theorem}
\begin{proof}
First we make the following observations.
	\begin{enumerate}
		\item The $i^{th}\operatorname{-}$row of $M_{(2k)\times (2k)}$ is unital if and only if the $i^{th}\operatorname{-}$row of \linebreak $M_{(2k)\times (2k)}.X$ is for any 
		$X\in SP_{2k}(\R)$.
		\item The conclusion holds for the matrix $M_{(2k)\times (2k)}$ if and only if the conclusion holds for the matrix $M_{(2k)\times (2k)}.X$ for any $X\in SP_{2k}(\R)$.
		\item For any $1\leq i\leq 2k$, we can replace the $i^{th}\operatorname{-}$row $[m_{i,1}\ \ldots \ m_{i,2k}]$ of $M_{(2k)\times (2k)}$ by another unital row $[\ti{m}_{i,1}\ \ldots \ \ti{m}_{i,2k}]$ 
		such that $m_{i,j}\equiv \ti{m}_{i,j}\mod \mcl{I}_i, 1\leq j\leq 2k$.  
	\end{enumerate}
If $\mcl{I}=R$ then $\mcl{I}_i=\R,1\leq i\leq 2k$ we can simply take $N$ as the identity matrix. So we can assume that $\mcl{I}\sbnq \R$ and satisfies the USC. So if for any $1\leq i\leq 2k, \mcl{I}\subseteq\mcl{I}_i\sbnq \R$ then using Lemma~\ref{lemma:UnitalSuperIdeal}, $\mcl{I}_i$ also satisfies the USC.

We prove this theorem in several steps with the central idea being to transform $M_{(2k)\times (2k)}=[m_{i,j}]_{1\leq i,j\leq 2k}=\mattwo {A=[a_{i,j}]_{1\leq i,j\leq k}}{B=[b_{i,j}]_{1\leq i,j\leq k}}{C=[c_{i,j}]_{1\leq i,j\leq k}}{D=[d_{i,j}]_{1\leq i,j\leq k}}$ to another matrix for which the conclusion of the theorem holds.

\begin{enumerate}[label=Step(\Alph*):]
\item If $\us{i=1}{\os{k}{\prod}}\mcl{I}_i=\R$ then we replace $A$ by the identity matrix $I_{k\times k}$ and go to Step(H). Hence we assume that $\us{i=1}{\os{k}{\prod}}\mcl{I}_i\neq \R$.
\item If $\mcl{I}\subseteq \mcl{I}_1\sbnq \R$ then using Lemma~\ref{lemma:UnitalSuperIdeal}, we have that $\mcl{I}_1$ satisfies the USC. So we can make $m_{1,1}$ a unit modulo $\mcl{I}_1$ by applying an $SP_{2k}(\R)$ transformation as follows. If $t_1=a_{1,1}+\us{j=2}{\os{k}{\sum}}a_{1,j}e_{1,j}+\us{j=1}{\os{k}{\sum}}b_{1,j}f_{1,j}$ is a unit modulo $\I_1$ then using Proposition~\ref{prop:UnitRowExt} there exists a matrix 
$g_2=\mattwo{E}{0}{F}{(E^t)^{-1}}\in SP_{2k}(\R)$ such that the first column of $g_2$ is $(1,e_{1,2},\ldots,e_{1,k},f_{1,1},\ldots,f_{1,k})^t$. Here 
\equ{E=\mattwo{\mattwo{1}{0}{e_{1,2}}{1}}{0_{2\times (k-2)}}{\matthreetwo{e_{1,3}}{0}{\vdots}{\vdots}{e_{1,k}}{0}}{I_{(k-2)\times (k-2)}}\in SL_k(\R).}
The matrix $Mg_2$ has the property that $(Mg_2)_{11}=t_1$ is a unit modulo $\I_1$. So we assume that $m_{1,1}$ is a unit modulo $\I_1$.

Let $z_{1,1}\in \R$ be such that $z_{1,1}m_{1,1} \equiv 1 \mod \mcl{I}_1$.
Now we add $-z_{1,1}m_{1,j}$ times the $1^{th}\operatorname{-}$column to the $j^{th}\operatorname{-}$column for $1\leq j\leq k$ for the submatrix $A$ of $M$. The matrices used in these transformations are of the form 
\equ{\mattwo{U}{0_{k\times k}}{0_{k\times k}}{(U^t)^{-1}}, U\in SL_k(\R).}
The $1^{th}\operatorname{-}$row of $A$ becomes $[a_{1,1}\ \ldots \ a_{1,k}]$ with the following
properties.
\begin{itemize}
	\item $a_{1,1}=m_{1,1}$ is a unit modulo $\mcl{I}_1$ if $\mcl{I}_1\neq \R$.
	\item $a_{1,j}=m_{1,j}\in \mcl{I}_1$ for $2\leq j\leq k$.
	\item $\us{j=1}{\os{2k}{\sum}}(m_{1,j})=\R$, that is the $1^{th}\operatorname{-}$row of $M$ is a unital vector. In fact all rows of $M$ are unital.
\end{itemize} 
\item We inductively consider the $i^{th}\operatorname{-}$row of $M_{(2k)\times (2k)}$ for $1< i\leq k$.
Here if $\mcl{I}_i\neq \R$ and $a_{i,i}$ is not a unit modulo $\mcl{I}_i$ then we use Proposition~\ref{prop:bringunit} for the subscript $i$ and for the ideals $\mcl{I}=\mcl{I}_i$ and $\mcl{J}=\mcl{I}_1\ldots\mcl{I}_{i-1}$ which is co-maximal with $\mcl{I}_i$ to make $a_{i,i}$ a unit modulo $\mcl{I}_i$. 

In this procedure we get a unit $t_i$ modulo $\mcl{I}_i$ in the $(i,i)^{th}$-position as 
\equ{t_i=\us{j=1}{\os{i-1}{\sum}}e_{i,j}a_{i,j}+a_{i,i}+\us{j=i+1}{\os{k}{\sum}}e_{i,j}a_{i,j}+\us{j=1}{\os{k}{\sum}}f_{i,j}b_{i,j}} where $e_{i,j}\in \mcl{I}_1\ldots\mcl{I}_{i-1}$ for $j\neq i,1\leq j\leq k, f_{i,j}\in \mcl{I}_1\ldots\mcl{I}_{i-1}$ for $1\leq j\leq k$. 

Using Proposition~\ref{prop:UnitRowExt} there exists a matrix $g_3=\mattwo{E}{0}{F}{(E^t)^{-1}}\in SP_{2k}(\R)$ such that the $i^{th}$ column of $g_3$ is 
\equ{(e_{i,1},\ldots,e_{i,(i-1)},1,e_{i,(i+1)},\ldots,e_{i,k},f_{i,1},\ldots,f_{i,k})^t.} Here 
\equ{E=\mattwo{I_{(i-1)\times(i-1)}}{\matthreetwo{e_{i,1}}{0_{1\times (k-i)}}{\vdots}{\vdots}{e_{i,(i-1)}}{0_{1\times (k-i)}}}{0_{(k-i+1)\times (i-1)}}{\mattwo{1}{0_{1\times (k-i)}}
			{\matcolthree{e_{i,i+1}}{\vdots}{e_{i,k}}}{I_{(k-i)\times (k-i)}}}\in SL_k(\R)}
and the matrix entries in $F$ are contained in the ideal $\langle f_{i,1},\ldots,f_{i,k}\rangle \subseteq \mcl{I}_1\mcl{I}_2\ldots\mcl{I}_{i-1}$.

The matrix $Mg_3$ has the property that $(Mg_3)_{i,i}=t_i$ is a unit modulo $\I_i$. So we assume that $m_{i,i}=a_{i,i}$ is a unit modulo $\I_i$.
Now the matrix $A$ will have the following properties.
\begin{itemize}
	\item $a_{l,l}$ is a unit modulo $\mcl{I}_l$ for $0\leq l\leq i$ if $\mcl{I}_l\neq \R$.
	\item $a_{l,j}\in \mcl{I}_j$ for $j\neq l, 0\leq j\leq k, 0\leq l\leq (i-1)$ when Proposition~\ref{prop:bringunit} is applied appropriately, (that is, for the $i^{th}\operatorname{-}$row the subscript
	$i$ is chosen, the ideal $\mcl{I}=\mcl{I}_i$ is chosen, the ideal $\mcl{J}=\mcl{I}_1\ldots\mcl{I}_{i-1}$ is chosen in Proposition~\ref{prop:bringunit} and using the fact that the ideal 
	$\langle f_{i,1},\ldots,f_{i,k}\rangle \subseteq \mcl{I}_1\mcl{I}_2\ldots\mcl{I}_{i-1}$.
	\item All rows of $M_{(2k)\times (2k)}$ are unital. 
\end{itemize}

Let $z_{i,i}\in \R$ such that $z_{i,i}m_{i,i} \equiv 1 \mod \mcl{I}_i$.
Now we add $-z_{i,i}m_{i,j}$ times the $i^{th}\operatorname{-}$column to the $j^{th}\operatorname{-}$column for $1\leq j\leq k$ for the submatrix $A$ of $M$. The matrices used in these transformations are of the form 
\equ{\mattwo{U}{0_{k\times k}}{0_{k\times k}}{(U^t)^{-1}}, U\in SL_k(\R).}
Now the matrix $A$ will have the following properties.
\begin{itemize}
	\item $a_{l,l}$ is a unit modulo $\mcl{I}_l$ for $0\leq l\leq i$ if $\mcl{I}_l\neq \R$.
	\item $a_{l,j}\in \mcl{I}_j$ for $j\neq l, 0\leq j\leq k, 0\leq l\leq i$. Also when $l=i$.
	\item All rows of $M_{(2k)\times (2k)}$ are unital. 
\end{itemize}
\item We continue this procedure till the $k^{th}\operatorname{-}$row.
After this procedure we have the following properties for $A$.
\begin{itemize}
	\item The diagonal entry $a_{i,i}$ is a unit modulo $\mcl{I}_i$ if $\mcl{I}_i\neq \R$ for $1\leq i\leq k$.
	\item The non-diagonal entry $a_{i,j}\in \mcl{I}_i$ for $1\leq i\neq j\leq k$.
	\item All rows of $M_{2k\times 2k}$ are unital. 
\end{itemize}
\item Consider only the diagonal part of the matrix $A$, that is, the matrix $D_{k\times k}=\Diag(a_{1,1},\ldots,a_{k,k})$. 
We use Proposition~\ref{prop:diagdetone} to change the diagonal matrix to $D_{k\times k}$ to $\ti{D}_{k\times k}=\Diag(d_1,\ldots,d_k)$ such that we have 
$d_i\equiv a_{i,i}\mod \mcl{I}_i,1\leq i\leq k$ and $d_1\ldots d_k\equiv 1\mod \mcl{I}_1\ldots\mcl{I}_k$. The ideal $\mcl{I}_1\ldots\mcl{I}_k \neq \R$ by assumption. Hence 
\equ{\ol{\ti{D}}_{k\times k} \in SL_{k}\bigg(\frac{\R}{\mcl{I}_1\ldots \mcl{I}_k}\bigg).}

Since $\mcl{I}\subseteq\mcl{I}_1\ldots \mcl{I}_k\sbnq \R$,  the ideal $\mcl{I}_1\ldots \mcl{I}_k$ satisfies the USC. Hence by an application of the surjectivity theorem for the map $SL_k(\R) \lra SL_{k}\bigg(\frac{\R}{\mcl{I}_1\ldots\mcl{I}_k}\bigg)$
we conclude that there exists a matrix $\ti{A}=[\ti{a}_{i,j}]_{1\leq i,j\leq k}\in SL_{k}(\R)$ such that we have 
\equ{\ol{\ti{A}}=\ol{\ti{D}}_{k\times k} \in SL_{k}\bigg(\frac{\R}{\mcl{I}_1\ldots \mcl{I}_k}\bigg).}

\item We observe the following. 
\begin{itemize}
	\item The diagonal entries $\ti{a}_{i,i}\equiv d_i\equiv a_{i,i} \mod \mcl{I}_i$.
	\item The non-diagonal entries $\ti{a}_{i,j}\in \mcl{I}_i$ and hence $\ti{a}_{i,j}\equiv a_{i,j}$ for $0\leq i\neq j\leq k$.
\end{itemize}
So we replace the sub-matrix $A$ of $M$ by $\ti{A}$ with all the rows of $M$ being unital and hence we could as well assume that $M=\mattwo{A}{B}{C}{D}$ with $A\in SL_k(\R)$.

\item Now we multiply $M$ by $\mattwo{A^{-1}}{0}{0}{A^t}$ so the matrix $M$ becomes $\mattwo IBCD$ for some $B,C,D\in M_{k\times k}(\R)$ with all rows of $M$ being unital.

\item Now we make the matrix $B$ symmetric. For $1\leq i\neq j\leq k$ we have $\mcl{I}_i+\mcl{I}_j=\R$. So let $q_{i,j}\in \mcl{I}_i,q_{j,i}\in \mcl{I}_j$ be such that $q_{i,j}+q_{j,i}=1$. Now we solve for $b'_{i,j},b'_{j,i}\in \R$ such that $b_{i,j}+q_{i,j}b'_{i,j}=b_{j,i}+q_{j,i}b'_{j,i}$ for $1\leq i\neq j\leq k$. Replace the $(i,j)^{th},(j,i)^{th}$ entries of $B$ by $b_{i,j}+q_{i,j}b'_{i,j}=b_{j,i}+q_{j,i}b'_{j,i}$. The rows of $M$ are still unital after this replacement and now $B$ is symmetric.

\item Then we multiply $M$ by $\mattwo{I}{-B}{0}{I}\in SP_{2k}(\R)$ and so that $M=\mattwo I0CD$ for some $C,D\in M_{k\times k}(\R)$ and rows of $M$ being unital.

\item We have $M=\mattwo I0CD$. Now we transform $M$ as $M=\mattwo {U}{0}{C}{I}$ with $U\in SL_k(\R),C$ symmetric, $I\in SL_k(\R)$ is the identity matrix. If $\mcl{I}_{k+1}\ldots\mcl{I}_{2k}=\R$ then this step is easy. So we assume that $\mcl{I}_{k+1}\ldots\mcl{I}_{2k}\neq \R$.
For any $1\leq i\leq k$, we successively make $c_{i,i}$ a unit modulo $\mcl{I}_{k+i}$ if $\mcl{I}_{k+i} \neq \R$ as follows. Since $\mcl{I}_{k+i}$ satisfies the USC, let 
\equ{t_{k+i}=\us{j=1}{\os{i-1}{\sum}}c_{i,j}e_{i,j}+c_{i,i}+\us{j=i+1}{\os{k}{\sum}}c_{i,j}e_{i,j}+\us{j=1}{\os{k}{\sum}}d_{i,j}f_{i,j}} be a unit modulo $\mcl{I}_{k+i}$ with 
$e_{i,j}\in \mcl{I}_{k+1}\ldots\mcl{I}_{k+i-1},j\neq i, 1\leq j\leq k,f_{i,j}\in \mcl{I}_{k+1}\ldots\mcl{I}_{k+i-1},1\leq j\leq k$ using an application of Proposition~\ref{prop:bringunit} appropriately.
Then multiply $M$ by the matrix $\mattwo E0F{(E^{t})^{-1}}$ whose $i^{th}$-column is  $(e_{i,1},\ldots,e_{i,(i-1)},1,e_{i,(i+1)},\ldots,e_{i,k},f_{i,1},\ldots,f_{i,k})^t$ where 
\equ{E=\mattwo{I_{(i-1)\times (i-1)}}{\matthreetwo{e_{i,1}}{0_{1\times (k-i)}}{\vdots}{\vdots}{e_{i,(i-1)}}{0_{1\times (k-i)}}}{0_{(k-i+1)\times (i-1)}}
	{\mattwo{1}{0_{1\times (k-i)}}{\matcolthree{e_{i,(i+1)}}{\vdots}{e_{i,k}}}{I_{(k-i)\times (k-i)}}}\in SL_{k}(\R)}
and the matrix entries of $F$ are in the ideal $\langle f_{i,1},f_{i,2},\ldots,f_{i,k}\rangle \subseteq \mcl{I}_{k+1}\ldots\mcl{I}_{k+i-1}$. 

Since $\langle c_{i,i}\rangle +\mcl{I}_{k+i}=\R$ and $\mcl{I}_1\ldots\mcl{I}_k+\mcl{I}_{k+i}=\R$, we have $c_{i,i}\mcl{I}_1\ldots\mcl{I}_k+\mcl{I}_{k+i}=\R$. Hence there exists $s_{i,i}\in\mcl{I}_1\ldots\mcl{I}_k$ such that $c_{i,i}s_{i,i}+d_{i,i}\equiv 1\mod \mcl{I}_{k+i}$ for $1\leq i\leq k$. Now we multiply $M=\mattwo U0CD$ by a matrix of the form 
\equ{\mattwo I{S=\Diag(s_{1,1},s_{2,2},\ldots,s_{k,k})}{0}{I}} to obtain $\mattwo {U}{US}{C}{CS+D}$. Then replace $US$ by zero and rows of $M$ are still unital. We have $M$ is of the form $\mattwo {U}{0}{C}{D}$ such that $d_{i,i}\equiv 1\mod \mcl{I}_{k+i}$, $1\leq i\leq k$. Replace $d_{i,i}$ by $1$ and now rows of $D$ are unital. 

Now for $k>1$ using Theorem~\ref{theorem:GenSurjMainGenIdeals} there exists a matrix $V\in SL_k(\R)$ such that $v_{i,j}\equiv d_{i,j}\mod \mcl{I}_{k+i}$ for $1\leq i,j\leq k$. Now replace $D$ by $V$ in $M$ and we have that the rows of $M$ are unital. Then right multiply $M$ by $\mattwo {V^t}00{V^{-1}}$ to obtain a matrix of the form $\mattwo {U}{0}{C}{I}$ with $U\in SL_k(\R)$.

Now we make the submatrix $C$ symmetric in $M=\mattwo UOCI$ similar to Step (H). 

\item Then right multiply $M=\mattwo U0CI$ by $\mattwo I0{-C}I$ to obtain a matrix of the form $\mattwo U00I$.  

\item Now we transform $M$ into identity matrix. Let $V\in SL_k(\R)$ such that $V\equiv U^{-1} \mod \mcl{I}_1\ldots\mcl{I}_k$ and $V\equiv I \mod \mcl{I}_{k+1}\ldots\mcl{I}_{2k}$. Such a matrix $V$ exists because the map
\equa{SL_k(\R)&\lra SL_k\bigg(\frac{\R}{\mcl{I}_1\ldots\mcl{I}_{2k}}\bigg)\\ &\lra SL_k\bigg(\frac{\R}{\mcl{I}_1\ldots\mcl{I}_{k}}\bigg)\oplus SL_k\bigg(\frac{\R}{\mcl{I}_{k+1}\ldots\mcl{I}_{2k}}\bigg)}
is surjective. Right multiply $M$ by $\mattwo{V}{0}{0}{(V^t)^{-1}}$ and replace $M$ by identity matrix in $SP_{2k}(\R)$.
\end{enumerate}
This completes the proof of Theorem~\ref{theorem:GenSurjMain}.
\end{proof}
Now we prove the fourth main Theorem~\ref{theorem:FullGenSurjOne} of the article.
\begin{proof}[Proof of Theorem~\ref{theorem:FullGenSurjOne}]
	Let \equa{\big([a_{1,1}:a_{1,2}:\ldots: a_{1,2k}],[a_{2,1}:a_{2,2}:\ldots: a_{2,2k}],\ldots,&[a_{2k,1}:a_{2k,2}:\ldots: a_{2k,2k}]\big)\\&\in \us{i=1}{\os{2k}{\prod}}\mbb{PF}^{2k-1,(m^i_1,m^i_2,\ldots,m^i_{2k})}_{\I_i}.}
	Consider $A_{2k\times 2k}=[a_{i,j}]_{1\leq i,j\leq 2k}\in M_{2k\times 2k}(\R)$ for which Theorem~\ref{theorem:GenSurjMain} can be applied. Therefore we get 
	$B=[b_{i,j}]_{1\leq i,j\leq 2k} \in SP_{2k}(\R)$ such that $b_{i,j}\equiv a_{i,j} \mod \mcl{I}_i,1\leq i,j\leq 2k$. Hence we get for $1\leq i\leq 2k$,
	\equ{[b_{i,1}:b_{i,2}:\ldots: b_{i,2k}]=[a_{i,1}:a_{i,2}:\ldots: a_{i,2k}]\in \mbb{PF}^{2k-1,(m^i_1,m^i_2,\ldots,m^i_{2k})}_{\I_i}.}
	This proves main Theorem~\ref{theorem:FullGenSurjOne}.
\end{proof}
We end this section with a final remark.
\begin{remark}
We have seen in Theorem~\ref{theorem:GenCRTSURJ} that the map in Question~\ref{ques:GenCRTSURJ} is indeed bijective if the either ideal $\mcl{I}$ satisfies the USC or $\mcl{I}=\R$ with no conditions on the commutative ring $\R$ with unity and for all values $m_j\in \N,0\leq j\leq l$. The answer to Question~\ref{ques:GenCRTSURJ} in a greater generality is not known.

We have seen in Theorem~\ref{theorem:FullGenSurj} that the map in Question~\ref{ques:ProjHighDim} is indeed surjective if the product ideal $\mcl{I}=\us{i=0}{\os{k}{\prod}}\mcl{I}_i$ satisfies the USC or if $\mcl{I}=\R$. There are no conditions required for the type of ring and the Theorem~\ref{theorem:FullGenSurj}  holds for all values $m^j_i\in \N, 0\leq i,j\leq k$.
The answer to Question~\ref{ques:ProjHighDim} in a greater generality is not known. As stated in Question~\ref{ques:ProjHighDim}, the map is not surjective because of Example~\ref{Example:NotSurj}. 

Again if the product ideal $\mcl{I}=\us{i=1}{\os{2k}{\prod}}\mcl{I}_i$ satisfies the USC or if $\mcl{I}=\R$ then using Theorem~\ref{theorem:FullGenSurjOne} , we conclude surjectivity of the map in Question~\ref{ques:ProjHighDimOne} with no conditions on the ring $\R$ and for all values $m_j^i\in \N,1\leq i,j\leq 2k$. The answer to Question~\ref{ques:ProjHighDimOne} in a greater generality is not known.
\end{remark}

\section{\bf{The Classical Groups $O_n(\Z), O_{(p,q)}(\Z),UU_n(\Z)$ over Integers}}
Here we give examples of certain classical groups over the integers where the surjectivity fails. 
Let $n\in \mbb{N}$ and \equ{O_n(\Z)=\{P\in M_n(\Z)\mid P^tP=PP^t=I_{n\times n}\}.}
Then it is a finite group consisting of $n\times n$ monomial matrices with $\pm 1$ non-zero entries of determinant $\pm 1$. Hence it is immediate that surjectivity fails.
For $n\in \N$ and \equ{UU_n(\Z)=\{P\in SL_n(\Z)\mid P \text{ is a unipotent upper triangular matrix}\}} again the surjectivity fails. Now we prove a theorem.
\begin{theorem}
	Let $p,q\in \mbb{N}, J=\mattwo{I_{p\times p}}{0_{p\times q}}{0_{q \times p}}{-I_{q\times q}}$
	and \equ{O_{(p,q)}(\Z)=\{A\in M_{p+q}(\Z)\mid A^tJA=J\}.}
	Let $r_i,1\leq i\leq (p+q)$ be distinct primes. Then the maps \equ{O_{p+q}(\Z) \lra \us{i=1}{\os{p+q}{\prod}}\mbb{PF}^{p+q-1}_{r_i}} given by 
	\equ{A=[a_{i,j}]_{i,j=1,\ldots, p+q}\lmt \big([a_{1,1}:\ldots:a_{1,(p+q)}],\ldots,[a_{(p+q),1}:\ldots:a_{(p+q),(p+q)}]\big)}
	or 
	\equ{A=[a_{i,j}]_{i,j=1,\ldots, p+q}\lmt \big([a_{1,1}:\ldots:a_{(p+q),1}],\ldots,[a_{1,(p+q)}:\ldots:a_{(p+q),(p+q)}]\big)}
	are both not surjective.
\end{theorem}
\begin{proof}
	If $A^tJA=J$ then we have $\Det(A)=\pm 1$ and $(A^{-1})^tJA^{-1}=J$ and moreover $A^{-1}=JA^tJ$. So if $A=[a_{i,j}]_{i,j=1,\ldots, p+q}\in O_{(p,q)}(\Z)$ then we have 
	\begin{itemize}
		\item for each $1\leq i\leq p, \us{j=1}{\os{p}{\sum}}a_{i,j}^2-\us{j=p+1}{\os{p+q}{\sum}}a_{i,j}^2=1$.
		\item for each $p+1\leq i\leq p+q, \us{j=1}{\os{p}{\sum}}a_{i,j}^2-\us{j=p+1}{\os{p+q}{\sum}}a_{i,j}^2=-1$.
		\item for each $1\leq j\leq p, \us{i=1}{\os{p}{\sum}}a_{i,j}^2-\us{i=p+1}{\os{p+q}{\sum}}a_{i,j}^2=1$.
		\item for each $p+1\leq j\leq p+q, \us{i=1}{\os{p}{\sum}}a_{i,j}^2-\us{i=p+1}{\os{p+q}{\sum}}a_{i,j}^2=-1$.
	\end{itemize}
	Now consider a non-square $\gl\in \mbb{F}_{r_1},\square\neq \gl$.
	Let $[a_1:a_2:\ldots:a_{p+q}]\in \mbb{PF}^{p+q-1}_{r_1}$ such that $a_1^2+\ldots +a_p^2-a_{p+1}^2-\ldots-a_{p+q}^2=\gl$. Then there does not exist $A=[a_{i,j}]_{i,j=1,\ldots, p+q}\in O_{(p,q)}(\Z)$ such that 
	$[a_{1,1}:\ldots:a_{1,(p+q)}]=[a_1:a_2:\ldots:a_{p+q}]$. If it exists then there exists $\gm\in \mbb{F}^{*}_{r_1}$ such that $\gm a_{1,i}=a_i,1\leq i\leq (p+q) \Ra \gm^2=\gl$ which is a contradiction.
	Hence both the row and column maps fail to be surjective.
\end{proof}

\section{\bf{Appendix}}
In this section we prove certain facts about zariski density of integral points and (usual/zariski) topological properties of certain complex affine algebraic sets.
But first we prove a lemma.
\begin{lemma}
	\label{lemma:ZD}
	Let $X$ be an affine or projective algebraic set over an infinite field $\mbb{F}$. 
	Let $A\subs X$. Let $C\subs X$ be an irreducible curve. If $A\cap C$ is infinite then the zariski closure $\ol{A}$ in the affine or projective space of the set $A$ contains the curve $C$. 
\end{lemma}
\begin{proof}
	The curve $C$ has co-finite topology. If $D$ is zariski closed and $D\supseteq A$ then $D\supseteq (A\cap C)$. Hence $D\supseteq C$. This proves the lemma. 
\end{proof}

\subsection{\bf{Equation $x^2+y^2-z^2=1$: Zariski Density}}

Here in this section we prove zariski density of integral points $X(\Z)$ in $X(\mbb{C})$ where $X(\Z)$ and $X(\mbb{C})$ are as stated in the lemma below.
\begin{lemma}
	Let $X(\mbb{C})=\{(x,y,z)\in \mbb{C}^3\mid x^2+y^2-z^2=1\}$ and $X(\Z)=\{(x,y,z)\in \Z^3\mid x^2+y^2-z^2=1\}$. Then $X(\Z)$ is zariski dense in $X(\mbb{C})$.
\end{lemma}
\begin{proof}
	First we observe that the set $T=\bigg\{\mattwo abbc\in M_2(\Z) \mid ac=b^2-1, b \text{ even }\bigg\}$ is zariski dense in the set of symmetric matrices 
	$S=\bigg\{\mattwo abbc \in M_2(\mbb{C})\mid ac=b^2-1\bigg\}$. Let $a$ be an odd integer. For any odd integer $d$ the point $(d(cd+2),cd+1)$ satisfies the equation $xc=y^2-1$ and $cd+1$ is an even integer. Hence 
	$\ol{T}\supseteq \{\mattwo abbc \in M_2(\mbb{C})\mid ac=b^2-1,c \text{ an odd integer}\}$ using Lemma~\ref{lemma:ZD}.
	Similarly $\ol{T}\supseteq \{\mattwo abbc \in M_2(\mbb{C})\mid ac=b^2-1,a \text{ an odd integer}\}$.
	Let $\mattwo {a_0}{b_0}{b_0}{c_0}\in M_2(\mbb{C})$ such that $a_0c_0=b_0^2-1,b_0\neq \pm 1$. 
	The irreducible curve for the fixed $b_0\in \mbb{C}$ given by $ac=b_0^2-1$ has infinitely many points $(a,c)$ where $a$ or $c$ is an odd integer.
	Hence $\mattwo {a_0}{b_0}{b_0}{c_0}\in \ol{T}$. This proves that $\ol{T}=S$. Now consider the polynomial isomorphism 
	\equ{\psi:X(\mbb{C}) \lra S \text{ given by }\psi: (x,y,z) \lmt \mattwo {z+x}{y}{y}{z-x}.} 
	We observe that $\psi (X(\Z)) \sups T$. Hence $X(\Z)$ is zariski dense in $X(\mbb{C})$.
\end{proof}
\subsection{\bf{On Some Topology of the Classical Groups $GL_2(\mbb{C}),SL_2(\mbb{C})=SP_2(\mbb{C})$, $SO_{(2,1)}(\mbb{C})\cong SO_3(\mbb{C})$}}
We mention some topological properties of classical groups in this section. For which group, the surjectivity holds and the properties of zariski denseness and simply-connectedness can be compared.
\begin{theorem}[On Zariski Density and Simply-connectedness]
	~\\
	\begin{enumerate}
		\item $SL_2(\Z)=SP_2(\Z)$ is zariski dense in $SL_2(\mbb{C})=SP_2(\mbb{C})$ and $SL_2(\mbb{C})=SP_2(\mbb{C})$ is simply-connected. 
		\item $GL_2(\Z)$ is not zariski dense in $GL_2(\mbb{C})$ and $GL_2(\mbb{C})$ is not simply-connected. 
		\item $SO_{(2,1)}(\mbb{\Z})$ is zariski dense in $SO_{(2,1)}(\mbb{C})\cong SO_3(\mbb{C})$ and $SO_{(2,1)}(\mbb{C})\cong SO_3(\mbb{C})$ is not simply-connected.
		\item $SO_3(\Z)$ is not zariski dense in $SO_3(\mbb{C})$ and $SO_3(\mbb{C})$ is not simply-connected.
		\item For $n\in \N$, the unipotent upper triangular group $UU_n(\Z)$ is zariski dense in $UU_n(\mbb{C})$ and the group $UU_n(\mbb{C})$ is simply connected. 
	\end{enumerate}
\end{theorem}
\begin{proof}
	First we prove $(1)$. Let $G$ be the closure of $SL_2(\Z)$ in $\mbb{C}^4$. Then $G\subseteq SL_2(\mbb{C})$. Consider the subset 
	$S=\bigg\{\mattwo abcd \mid ad-bc=1,a\neq 0\bigg\}\subs SL_2(\mbb{C})$. The set $S$ is homeomorphic in the zariski subspace topology to the set $T=\mbb{C}^{*}\times \mbb{C}^2\subs \mbb{C}^3$
	in the zariski subspace topology. The isomorphism is given by \equ{S \lra T \text { is given by } \mattwo abc{d=\frac{1+bc}{a}} \lra (a,b,c).}
	Now using Lemma~\ref{lemma:ZD} repeatedly for lines starting from pairs of relatively prime integers, 
	we conclude that the set \equa{D&=\{(a,b,c)\in \Z^3\mid \text{ there exists } d\in \Z \text{ with }ad-bc=1,a\neq 0\}\\
		&\subs \Z\bs \{0\}\times \Z\times \Z}
	is zariski dense in $T$. We observe the following.
	\begin{itemize}
		\item First $\ol{D} \supseteq \{(a,b,y)\in \Z^2\times \mbb{C}\mid gcd(a,b)=1\}$ and
		\item $\ol{D} \supseteq \{(a,x,c)\in \Z \times \mbb{C} \times \Z\mid gcd(a,c)=1\}$.
		\item Then $\ol{D} \supseteq \{(a,x,y)\in \Z\times \mbb{C}^2, a\neq 0\}$.
		\item Hence $\ol{D}\supseteq \mbb{C}^{*}\times \mbb{C}^2$.
	\end{itemize}
	
	So $G\sups S$. Similarly $G \sups \bigg\{\mattwo abcd \mid ad-bc=1,b\neq 0\bigg\}$. So $G=SL_2(\mbb{C})$. 
	
	Now we have a homeomorphism $SL_2(\mbb{C}) \lra SU_2(\mbb{C}) \times \mbb{R}^{+} \times \mbb{C}$ given by 
	\equ{\mattwo abcd \lmt \bigg(\mattwo {\frac a{\sqrt{\mid a \mid^2+\mid c \mid^2}}}{\frac {-\ol{c}}{\sqrt{\mid a \mid^2+\mid c \mid^2}}}{\frac c{\sqrt{\mid a \mid^2+\mid c \mid^2}}}
		{\frac {\ol{a}}{\sqrt{\mid a \mid^2+\mid c \mid^2}}},(\sqrt{\mid a \mid^2+\mid c \mid^2}),z\bigg)}
	where $z=\frac{b(\mid a \mid^2+\mid c \mid^2)+\ol{c}}{a}$ if $a\neq 0$ or $z=\frac{d(\mid a \mid^2+\mid c \mid^2)-\ol{a}}{c}$ if $c\neq 0$
	and we also have 
	\equ{\mattwo abcd =\mattwo {\frac a{\sqrt{\mid a \mid^2+\mid c \mid^2}}}{\frac {-\ol{c}}{\sqrt{\mid a \mid^2+\mid c \mid^2}}}{\frac c{\sqrt{\mid a \mid^2+\mid c \mid^2}}}
		{\frac {\ol{a}}{\sqrt{\mid a \mid^2+\mid c \mid^2}}}\mattwo {\sqrt{\mid a \mid^2+\mid c \mid^2}}{\frac z{\sqrt{\mid a \mid^2+\mid c \mid^2}}}{0}{\frac 1{\sqrt{\mid a \mid^2+\mid c \mid^2}}}.}
	So the fundamental group of $SL_2(\mbb{C})$ is trivial since $SU_2(\mbb{C})$ is homeomorphic to $S^3$.
	
	Now we prove $(2)$. We have $GL_2(\Z)=SL_2(\Z)\sqcup \mattwo 100{-1}SL_2(\Z)$. Hence its closure using $(1)$ is 
	\equ{\bigg(SL_2(\mbb{C})\bigsqcup \mattwo 100{-1}SL_2(\mbb{C})\bigg)=\{A\in GL_2(\mbb{C})\mid \Det(A)=\pm 1\}.}
	So $GL_2(\Z)$ is not zariski dense in $GL_2(\mbb{C})$. Also $U_2(\mbb{C})$ is a strong deformation retract of $GL_2(\mbb{C})$ by the Gram-Schmidt orthonormalization process.  
	The group $SU_2(\mbb{C})\times \mbb{R}$ is a universal covering map (with respect to usual topology) of $U_2(\mbb{C})$ via the map $(A,t) \lra e^{\gp it}A$ with a discrete non-trivial kernel isomorphic to $\Z$.
	So $GL_2(\mbb{C})$ is not simply-connected.
	
	Now we prove $(3)$. Consider the polynomial map $Q:SL_2(\mbb{C}) \lra SO_{(2,1)}(\mbb{C})$ given by 
	\equ{\mattwo abcd \lmt \matthree{\frac{a^2-b^2-c^2+d^2}{2}}{ac-bd}{\frac{a^2-b^2+c^2-d^2}{2}}{ab-cd}{bc+ad}{ab+cd}{\frac{a^2+b^2-c^2-d^2}{2}}{ac+bd}{\frac{a^2+b^2+c^2+d^2}{2}}.}
	This map is an universal covering (with respect to usual topology) with kernel $\bigg\{\pm \mattwo 1001\bigg\}$. So $SO_{(2,1)}(\mbb{C})$ is not simply-connected. 
	Consider the set \equ{P=\bigg\{\mattwo abcd\in SL_2(\Z)\mid \text{three of }a,b,c,d \text{ are not odd integers}\bigg\}.}
	For example the matrix $\mattwo 2153\nin P$. Then $P$ is also zariski dense in $SL_2(\mbb{C})$. The image $Q(P)$ of $P$ is contained in $SO_{(2,1)}(\Z)$ and is zariski dense in $SO_{(2,1)}(\mbb{C})$ since we have 
	$Q(\ol{P})\subs \ol{Q(P)}$ in zariski topology. So $SO_{(2,1)}(\Z)$ is zariski dense in $SO_{(2,1)}(\mbb{C})$.
	
	We exhibit an isomorphism between $SO_{(2,1)}(\mbb{C})$ and $SO_3(\mbb{C})$ given by
	\equ{\matthree rstuvwxyz \lra \matthree rs{-it}uv{-iw}{ix}{iy}{z}.}
	Now we prove $(4)$. $SO_3(\mbb{C})$ is not simply-connected as we have seen that $SO_{(2,1)}(\mbb{C})$ is not so. Also $SO_3(\Z)$ is a finite group and hence it is already zariski closed.
	
	The assertion $(5)$ is also clear.
	This completes the proof of the theorem.  
\end{proof}
\subsection{\bf{Equation $x^2+y^2-z^2=1$: Simply-connectedness}}

Here in this section we prove simply-connectedness of $X(\mbb{C})=\{(x,y,z)\in \mbb{C}^3\mid x^2+y^2-z^2=1\}$.
\begin{lemma}
	Let $X(\mbb{C})=\{(x,y,z)\in \mbb{C}^3\mid x^2+y^2-z^2=1\}$. Then $X(\mbb{C})$ is simply-connected.
\end{lemma}
\begin{proof}
	First we note that $X(\mbb{C})$ is homeomorphic to the two dimensional complex sphere $S^2_{\mbb{C}}=\{(x,y,z)\in \mbb{C}^3\mid x^2+y^2+z^2=1\}$. So we consider $S^2_{\mbb{C}}$ instead of $X(\mbb{C})$.
	We observe that $SO_3(\mbb{C})$ acts transitively on $S^2_{\mbb{C}}$ and the stabilizer of $e_3^3=(0,0,1)$ is $SO_2(\mbb{C})$ isomorphic to 
	\equ{\bigg\{\matthree {a}{-b}{0}{b}{a}{0}{0}{0}{1}\in M_3(\mbb{C})\mid a^2+b^2=1\bigg\}\subs SO_3(\mbb{C}).}
	So \equ{X(\mbb{C})\cong S^2_{\mbb{C}}\cong \frac{SO_3(\mbb{C})}{SO_2(\mbb{C})}.}
	We also observe that $SO_2(\mbb{C})=\{\mattwo a{-b}ba\in M_2(\mbb{C})\mid a^2+b^2=1\}\subs SL_2(\mbb{C})$.
	Now we proceed one step further which is useful and prove that in fact  
	\equ{X(\mbb{C})\cong S^2_{\mbb{C}}\cong \frac{SL_2(\mbb{C})}{SO_2(\mbb{C})}.}
	This is because the double covering map $SL_2(\mbb{C})\lra SO_3(\mbb{C})$ given by 
	\equ{\mattwo abcd \lmt \matthree{\frac{a^2-b^2-c^2+d^2}{2}}{ac-bd}{i(\frac{a^2-b^2+c^2-d^2}{2})}{ab-cd}{bc+ad}{i(ab+cd)}{-i(\frac{a^2+b^2-c^2-d^2}{2})}{-i(ac+bd)}{\frac{a^2+b^2+c^2+d^2}{2}}}
	maps $SO_2(\mbb{C})\subs SL_2(\mbb{C})$ onto $SO_2(\mbb{C})\subs SO_3(\mbb{C})$  given by 
	\equ{\mattwo a{-b}ba \lmt \matthree {a^2-b^2}{-2ab}0{2ab}{a^2-b^2}0{0}{0}{1}}
	which is $2:1$ map of $SO_2(\mbb{C})$ onto itself. 
	The group $SL_2(\mbb{C})$ is simply connected. The group $SO_2(\mbb{C})\cong V(xy-1)\cong \mbb{C}^{*}$ and hence the fundamental group is $\Z$.
	
	With this fiber bundle $SO_2(\mbb{C}) \hookrightarrow SL_2(\mbb{C}) \lra X(\mbb{C})$ we have the long exact sequence of homotopy groups given by 
	\equ{\lra \gp_1(SO_2(\mbb{C})) \lra \gp_1(SL_2(\mbb{C})) \lra \gp_1(X(\mbb{C}))\lra \gp_0(SO_2(\mbb{C})) \lra }
	Since $SO_2(\mbb{C})$ is path connected we have $\gp_1(X(\mbb{C}))=0$ and hence $X(\mbb{C})$ is simply connected.
\end{proof}
\begin{note}
	If $F \hookrightarrow E \os{p}{\lra} B$ is a fiber bundle with path connected fiber $F$ then any loop in $B$ based at $b\in B$ can be lifted to a path whose end points can be joined in the same fiber $p^{-1}(b)\cong F$ and be made a loop. Hence if $E$ is simply-connected and $F$ is path connected then $B$ is simply-connected.
\end{note}

\bibliographystyle{abbrv}
\def\cprime{$'$}

\end{document}